\documentclass[12pt]{amsart}
\usepackage{amssymb,amscd,verbatim}
\usepackage[all]{xy}
\usepackage[colorlinks,linkcolor=blue,citecolor=blue,urlcolor=red]{hyperref}
\usepackage{aurical}
\usepackage[T1]{fontenc}
\usepackage[symbol]{footmisc}

\usepackage{color}

\swapnumbers
\newtheorem{thm}{Theorem}[subsection]
\newtheorem{propose}[thm]{Proposition}
\newtheorem{lemma}[thm]{Lemma}
\newtheorem{cor}[thm]{Corollary}

\theoremstyle{definition}
\newtheorem{defn}[thm]{Definition}
\newtheorem{axiom}[thm]{Axiom}
\newtheorem{remark}[thm]{Remark}
\newtheorem{remarks}[thm]{Remarks}
\newtheorem{example}[thm]{Example}
\newtheorem{examples}[thm]{Examples}

\newcommand{\N}{\mathbb{N}}
\newcommand{\Z}{\mathbb{Z}}     
\newcommand{\Q}{\mathbb{Q}}     
 
\newcommand{\Hom}{\operatorname{\underline{\sf Hom}}} 
\newcommand{\Coh}{\operatorname{\underline{\sf Coh}}} 
\newcommand{\Dom}{{\operatorname{\sf Dom}}}  
\newcommand{\Add}{\operatorname{\sf Add}}      
 
\newcommand{\Ex}{\operatorname{\sf Ex}}      
\newcommand{\Lex}{\operatorname{\sf Lex}}      
\newcommand{\Cat}{\operatorname{\sf Cat}}

\newcommand{\Ext}{\operatorname{Ext}}      
          
\newcommand{\Graph}{\operatorname{Graph}}  
\newcommand{\Pro}{\operatorname{Pro}}

\newcommand{\Rep}{\operatorname{\sf Rep}}
\newcommand{\Psh}{\operatorname{Pshv}}

\newcommand{\Shv}{\operatorname{Shv}}
\newcommand{\Ch}{\operatorname{Ch}}

\newcommand{\SSet}{\operatorname{SSet}}

\newcommand{\T}{\mathbb{T}} 
\newcommand{\TMod}{\text{$\T$-{\sf Mod}}}
\newcommand{\RMod}{\text{\rm $R$-{Mod}}}
\newcommand{\Rmod}{\text{\rm $R$-{mod}}}

\newcommand{\Kmod}{\text{\rm $K$-{mod}}}
\newcommand{\ZMod}{\text{\rm $\Z$-{Mod}}}
\newcommand{\Zmod}{\text{\rm $\Z$-{mod}}}

\newcommand{\Qmod}{\text{\rm $\Q$-{mod}}}
\newcommand{\iRMod}{\text{\rm $R$-{\underline{Mod}}}}
\newcommand{\iRmod}{\text{\rm $R$-{\underline{mod}}}}
\newcommand{\iZMod}{\text{\rm $\Z$-{\underline{Mod}}}}
\newcommand{\iZmod}{\text{\rm $\Z$-{\underline{mod}}}}
\newcommand{\iKmod}{\text{\rm $K$-{\underline{mod}}}}

\newcommand{\Ind}{\operatorname{Ind}}
\newcommand{\Ab}{\operatorname{Ab}}
\newcommand{\fp}{\operatorname{fp}}

\newcommand{\End}{\operatorname{End}}      

\newcommand{\im}{\operatorname{Im}}        
\renewcommand{\ker}{\operatorname{Ker}}  
\newcommand{\coker}{\operatorname{Coker}} 
\newcommand{\gr}{\operatorname{gr}}        

\newcommand{\tr}{{\operatorname{tr}}}        

\newcommand{\by}[1]{\stackrel{#1}{\rightarrow}}
\newcommand{\longby}[1]{\stackrel{#1}{\longrightarrow}}

\renewcommand{\tilde}{\widetilde}
\newcommand{\df}{\mbox{\,${:=}$}\,}
\newcommand{\ie}{{\it i.e. }}
\newcommand{\cf}{{\it cf. }}
\newcommand{\eg}{{\it e.g. }}

\newcommand{\into}{\hookrightarrow}
\newcommand{\onto}{\mbox{$\to\!\!\!\!\to$}}
\newcommand{\qto}{\rightleftarrows}
\renewcommand{\implies}{\mbox{$\Rightarrow$}}
\renewcommand{\c}{\mbox{\scriptsize{$\circ$}}} 

\newcommand{\limdir}[1]{\mathop{\rm
``lim"}_{\buildrel\longrightarrow\over{#1}}}
\newcommand{\limd}[1]{\mathop{\rm
lim}_{\buildrel\longrightarrow\over{#1}}}

\renewcommand{\lim}{\varprojlim}

\newcommand{\boxtensor}{\def\boxtimesten{\Box\kern-7.59pt\raise1.2pt
\hbox{$\times$} }}                                  

\newcounter{elno}                   

\newcommand{\cA}{\mathcal{A}}
\newcommand{\cB}{\mathcal{B}}
\newcommand{\cC}{\mathcal{C}}
\newcommand{\cD}{\mathcal{D}}

\newcommand{\cF}{\mathcal{F}}

\newcommand{\cP}{\mathcal{P}}

\newcommand{\cR}{\mathcal{R}}
\newcommand{\cS}{\mathcal{S}}

\renewcommand{\phi}{\varphi}
\renewcommand{\epsilon}{\varepsilon}

\title{Universal cohomology theories}
\author{Luca Barbieri-Viale}
\address{Dipartimento di Matematica ``F. Enriques'', Universit{\`a} degli Studi di Milano\\  Via C. Saldini, 50\\ I-20133 Milano\\ Italy}
\email{luca.barbieri-viale@unimi.it}

\begin{document}
\begin{abstract}
We furnish any category of a universal (co)homology theory. Universal (co)homologies and universal relative (co)homologies are obtained by showing representability of certain functors and take values in $R$-linear abelian categories of motivic nature, where $R$ is any commutative unitary ring. Universal homology theory on the one point category yields ``hieratic'' $R$-modules, \ie the indization of Freyd's free abelian category on $R$. Grothendieck $\partial$-functors and satellite functors are recovered as certain additive relative homologies on an abelian category for which we also show the existence of universal ones. 
\end{abstract}
\maketitle
\tableofcontents
\section{Introduction}
The main goal of this paper is to show that there is a ``simple'' (co)homological picture providing universal (co)homology theories in abelian categories, revisiting and developing the previously hinted construction of theoretical motives \cite{BV}. This unified framework for (co)homology theories on any fixed category $\cC$ with values in variable abelian categories $\cA$ is achieved through the solution of representability problems.

\subsection{A glimpse at the foremost theme} 
The general constructive method for universal theories can be sketched as follows: {\it i)} represent (co)homology theories in abelian categories through suitable diagrams, {\it ii)} show universal representations of such  diagrams and {\it iii)} impose ``relations'' in order to satisfy axioms. Note that, in principle, this method can be adapted to (co)homology theories in non-additive categories with suitable exactness properties. Moreover, it is conceivable and desirable that this framework shall also be revisited in the $(\infty, 1)$-category setting.

Notably, all these (co)homological constructions should have geometric and arithmetic applications, providing a new context and a precise formulation of universal homology theories for topological spaces, manifolds and schemes with values in corresponding abelian categories. 

Novelties already appear in the  case of the trivial category $\cC = {\bf 1}$ (here ${\bf 1}$ is the one point category) for which the model category of simplicial sets $\SSet$ provides the universal homotopy theory; in this case, for any (commutative unitary) ring $R$, we also have the universal $R$-linear abelian category $\Ab_R$ generated by $R$ regarded as a preadditive category: this is Freyd's universal (essentially small) $R$-linear abelian category on the point. This abelian category $\Ab_R$ is characterised by the following property: picking an object of an (essentially small) $R$-linear abelian category $\cA$, \ie  a functor  ${\bf 1} \to \cA$,  is equivalent to giving an $R$-linear exact functor from $\Ab_R$ to $\cA$. 

Note that $\Ab_R\neq \Rmod$ the category of finitely presented $R$-modules, in general. The indization $\Ind \Ab_R\cong\iRMod$ is equivalent to the Grothendieck category of ``hieratic'' $R$-modules, \ie $R$-linear additive functors from $\Rmod$ to the category  $\RMod$ of all $R$-modules. By the universal property, we obtain a canonical $R$-linear exact functor from ``hieratic'' $R$-modules to any Grothendieck category, such as all $R$-modules as well as ``condensed'' $R$-modules, sending the universal ``hieratic'' $R$-module to a generator.

For example, if $R=\Z$ is the ring of integers, we have that $\Ab_\Z\neq \Zmod$ but there is a natural exact quotient functor $\Ab_\Z\onto \Zmod$ whose $\Q$-linearization $\Ab_\Z\otimes\Q = \Qmod$ is given by finite dimensional $\Q$-vector spaces. 
The category of chain complexes $\Ch (\iZMod)$ also provides a refinement of $\Ch (\ZMod)$: the universal homotopy and homology theory on a point yields a Quillen pair
$$\SSet \qto \Ch (\iZMod)$$
refining the well known Quillen pair between simplicial sets and simplicial abelian groups. 

\subsection{Homology theories} 
The claimed universal homological picture can be achieved by constructing certain categories  whose objects are just homology or relative homology theories on a category $\cC$ with values in an abelian category $\cA$ at least. This $\cA$ could be required to have further structure and properties such as being $R$-linear or a Grothendieck category or a tensor category. The construction is clearly depending on the axioms that we choose to define such homology theories and these in turn depends on geometric properties of $\cC$.  The minimal data and conditions are that of $H =\{ H_i\}_{i\in \Z}$ a $\Z$-indexed family of functors $H_i : \cC\to \cA$. A morphism $\varphi: H \to H'$ between homology theories on $\cC$ with values in $\cA$ shall be given by a $\Z$-indexed family of natural transformations (see Definition \ref{homology}) yielding a category 
$$\Hom  (\cC, \cA)$$ of homologies. By composing with exact functors we can make up a 2-functor
$$\cA \leadsto  \Hom  (\cC, \cA): \Ex \to \Cat$$ 
from the 2-category $\Ex$ of abelian categories and exact functors to the (very large) 2-category of categories and functors.  A first key result is that $\Hom (\cC, - )$ is 2-representable by an abelian category of the same size of $\cC$: we simply denote by $\cA(\cC)$ the abelian category such that
$$\Hom (\cC, \cA)\cong \Ex (\cA(\cC), \cA)$$
is an equivalence, natural in $\cA$ (see Theorem \ref{pure}). If desired we can add up $R$-linearity and axioms that are preserved by $R$-linear exact functors and get ``decorated'' homological functors.  The construction of  $\cA(\cC)$ goes through the universal abelian representation 
theorem  \cite{BVP} (see Theorem \ref{unrep}) which is in turn based on Freyd's universal abelian category (see also \cite[Chap.\,4]{P} and \cite[Prop.\,12.4.1]{KH} for details). 

We shall refer to $\cA(\cC)$ as the $R$-linear \emph{abelian category generated by the homology theory} and $H\in \Hom (\cC, \cA(\cC))$ corresponding to the identity functor on $\cA(\cC)$ shall be the \emph{universal homology} on $\cC$. In fact, for any homology $H'\in \Hom (\cC, \cA)$ we have $r_{H'}: \cA(\cC)\to \cA$ the corresponding exact functor such that $r_{H'}(H) =H'$ is obtained by composition with  the ``realization'' functor $r_{H'}$. 

Adding axioms to the homology theory we obtain ``decorated'' abelian categories $\cA^\dag(\cC)$ given by taking quotients of $\cA(\cC)$. For example, adding the point axiom (see Axiom \ref{pointax}), for $\cC = {\bf 1}$ we obtain $\cA^{\rm point}({\bf 1})\cong \iRmod$ finitely presented ``hieratic'' $R$-modules (see Proposition \ref{point} and Example \ref{pointex}). For $\cC = N$ a group or a monoid regarded as a category with a point object we obtain the universal $R$-linear representation $h: R[N]\to \End (H)$ for $H\in \cA^{\rm point}(N)$ (see Example \ref{grprep}).
 
The analogue  cohomology theory is defined by $H =\{ H^i\}_{i\in \Z}$ a $\Z$-indexed family of functors $H^i : \cC^{op}\to \cA$ and denote $\Coh (\cC, \cA)$ the category of cohomologies. Define the opposite $H^{op}$ of an homology
 $H =\{H_i\}_{i\in \Z} : \cC\to \cA$ as $\{H^i\df H_{i}\}_{i\in\Z} : \cC^{op}\to \cA^{op}$ given by the opposite functors and note that $\Coh (\cC, \cA)= \Hom (\cC,\cA^{op})$.  As a consequence of the abstract duality (see Theorem \ref{dualrep}) the opposite category $\cA(\cC)^{op}$ is the \emph{abelian category generated by the cohomology theory} (see Corollary \ref{dualhom}) \ie it represents the 2-functor of cohomology theories 
$$\Coh (\cC, \cA)\cong \Ex (\cA(\cC)^{op}, \cA)$$  
and we obtain the \emph{universal cohomology} $H^{op}\in \Coh (\cC, \cA(\cC)^{op})$ as the opposite of $H\in \Hom (\cC, \cA(\cC))$ the universal homology. 

\subsection{Relative homology theories} 
Relative theories depend on the choice of a distinguished subcategory of $\cC$ in such a way that we can form the category of pairs $\cC^\square$ whose objects are distinguished morphisms of $\cC$ and morphisms are commutative squares of $\cC$. Denote $(X,Y)$ a distinguished morphism from $Y$ to $X$, for short, that is an object of $\cC^\square$. For a relative homology theory with values in abelian category $\cA$ we mean a collection of $\Z$-indexed functors 
$$(X, Y)\in \cC^\square \leadsto H_i (X, Y)\in \cA$$
satisfying the following long exact sequence in $\cA$ 
$$\cdots\to H_i(Y,Z) \to  H_i(X,Z) \to  H_i(X,Y) \longby{\partial_i} H_{i-1}(Y,Z)\to \cdots $$ 
for the triple $(X,Y)$, $(X,Z)$ and $(Y,Z)$ in $\cC^{\square}$, where $\partial_i: H_i(X,Y) \to H_{i-1}(Y,Z)$ shall be assumed to exists; furthermore, this long exact sequence is assumed to be natural in a canonical way (see Definition \ref{relative}). 
We then obtain a 2-functor
 $$\cA \leadsto \Hom (\cC^\square, \cA): \Ex \to \Cat$$
and we can show that is representable, obtaining an abelian category $\cA_\partial(\cC)$ along with
$$\Hom (\cC^\square, \cA)\cong \Ex (\cA_\partial(\cC), \cA)$$
 an equivalence of categories which is natural in $\cA$ (see Theorem \ref{mixed}). The construction of  $\cA_\partial(\cC)$ is also obtained as an application of  Freyd's universal abelian category via the universal abelian representation theorem  \cite{BVP} (see Theorem \ref{unrep}). For $\cC$ with an initial object we have a canonical exact functor 
$$r_\partial : \cA (\cC)\to\cA_\partial(\cC)$$
such that the essential image is generating (see Theorem \ref{genmixed}). 
We have a dual notion of relative cohomology and the 2-functor $\Coh (\cC^\square, -)$ of relative cohomologies is also representable: we obtain the universal cohomology as the opposite of universal homology with values in $\cA_\partial(\cC)^{op}$ the opposite abelian category (see Corollary \ref{dualrelhom}). 

Note that for $\cA = \ZMod$ the category of relative homologies $\Hom (\cC^\square, \ZMod)$ and that of relative cohomologies $\Coh (\cC^\square, \ZMod)$  are exact definable categories (in the sense of \cite[Chap.\,10]{P}). Moreover,  the 2-category of small abelian categories with exact functors is antiequivalent to the 2-category of definable additive categories (see \cite{PR}).
 
We shall refer to $\cA_\partial(\cC)$ as the abelian category \emph{generated by the relative homology theory} and  the \emph{universal relative homology} on $\cC^{\square}$ is given by $H\in \Hom (\cC^{\square}, \cA_\partial(\cC))$, corresponding to the identity functor on $\cA_\partial(\cC)$, and it shall be described by
$$H =\{H_i\}_{i\in \Z } : \cC^{\square}\to  \cA_\partial (\cC)$$
a family of functors. For any relative homology $H'\in \Hom (\cC^{\square}, \cA)$ we have $r_{H'}: \cA_\partial(\cC)\to \cA$ the corresponding exact  functor such that $r_{H'}(H) =H'$. By the way adding axioms to the relative homology theory we obtain ``decorated'' relative homologies and corresponding ``decorated'' abelian categories $\cA_\partial^\dag(\cC)$ which are quotients of  $\cA_\partial(\cC)$ such that
$$\Hom_\dag (\cC^\square, \cA)\cong \Ex^\dag (\cA_\partial^\dag(\cC), \cA)$$
is an equivalence, as it will be clearer in the following.
For example, adding the point axiom we obtain $\cA_\partial^{\rm point}({\bf 2})\cong \iRmod$ (see Example \ref{pointexrel}). We can apply this framework to Grothendieck $\partial$-functors: for $\cC = \cA$ an abelian category and $\cA^\square$ the category of monos we get a \emph{universal Grothendieck homology} (see Theorem \ref{unipart}) representing the 2-functor  $\Hom_{\partial}(\cA, -)$ given by Grothendieck homological functors regarded as $\partial$-homologies (Definition \ref{parthom}). 

Remark that these categories $\cA_\partial^\dag (\cC)$ shall be equivalent to \emph{theoretical motives} obtained as  quotients of the abelian category $\cA [\T]$ introduced in \cite{BV} (see Proposition \ref{Tmotives}). Moreover, for any $H'\in \Hom (\cC^{\square}, \cA)$ we can form the quotient $$\cA(H')\df \cA_\partial(\cC)/\ker r_{H'}$$ and we obtain the induced $H\in \Hom (\cC^{\square}, \cA (H'))$ given by the image of the universal homology under the projection functor; here we have that  $H: \cC^{\square}\to  \cA (H')$ is  the \emph{universal homology generated by $H'$}, \ie it is universal with respect to relative homologies which are comparable with $H'$ in an obvious sense. For example, the universal homology generated by singular homology, \eg the construction of Nori motives \cite{HMS} is based on the latter. 

Relative homology can be adorned by transfers yielding a universal relative homology with transfers (see Remark \ref{trans}). 
Furthermore, ordinary homology theory is a ``decoration'' of relative homology obtained by imposing Eilenberg-Steenrod axioms \cite{ES}:  we shall recover the classical topological setting in  \cite{UCTII} and we actually get universal ordinary homology with respect to an homological structure on any category representing ordinary homologies with values in abelian categories. 

Finally, we stress out that these constructions are in addition to Dugger's universal homotopy \cite{DU} and Voevodsky's homological triangulated category of a site with an interval \cite{Vh}: a canonical comparison with universal homotopy theories is given by passing to chain complexes.

\subsection{Historical commentary and notes} 
The axiomatic approach to homology theories, in particular to singular versus cellular homology, as it was introduced by Eilenberg and Steenrod in topology,  has been largely influential and their unicity theorem quite astonishing: the first key step in this story was taken around the years 1945--1952, see \cite{ES} and \cite{ESF}. Then a ramified study of topological generalised (co)homology theories emerged, see \cite{Dy}, for a start. 

A parallel history is that of Grothendieck construction of a Weil cohomology in algebraic geometry, which started from a wish-list of axioms and was afforded in the years 1958--1966 after two other key foundational steps: a first step in homological algebra, with the concept of satellite and that of $\partial$-functor, see \cite{CE} and \cite{Gr}, and a second step was the introduction of Grothendieck topologies. Notably, the stride from Weil cohomology to Grothendieck ``motives'' and ``motivic cohomology'', was originally based on a third foundational step,  the Tannakian formalism, but this approach to  ``motives'' is still dependent on the standard conjectures (\eg see \cite{DT} and \cite{Be}). However, broadly speaking,  the category of  ``motives'', in Grothendieck vision,  shall be regarded as a way to express a sort of abelian envelope of algebraic varieties and ``motivic cohomology'' shall be the abelian avatar of a variety.\footnote{Grothendieck, \emph{R\'ecoltes et Semailles}, note 59 on page P47: [Another way to see the category of motives over a field $k$ is to visualise it as a sort of ``enveloping abelian category'' of the category of separated  schemes which are of finite-type over $k$. The motive associated with such a scheme $X$ (or ``motivic cohomology of $X$'', which I denote $H^*_{mot} (X)$) thus appears as a kind of abelianised ``avatar'' of $X$.]}

Remark that Freyd \cite{Fr} also considered the following general question: given a category how nicely can it be represented in an abelian category? However, since its appearance around 1965, Freyd's universal abelian category of an additive category has not been imagined to be linked to the construction of  ``motives''.   Freyd also observed that there is an embedding of a triangulated category in an abelian category which is universal with respect to homological functors and then Heller \cite{AH} in 1968 constructed a universal homology in a stable abelian category (see Neeman \cite{NT} for a quintessential explanation). 

Actually, on the algebraic geometry side of the story, around 1997,  Nori provided a universal abelian category, making use of a variant of the Tannakian formalism (see \cite{HMS} for full details). Previously, around 1987, Deligne \cite{De} introduced the abelian category of mixed realisations and Andr\'e \cite{A}, on 1996,  proposed motivated cycles showing how to avoid the standard conjectures.  However, Nori's idea as well as Andr\'e and Deligne ideas of  ``motives'' -- being based on the Tannakian formalism -- makes use of existing fiber functors. Consequently, ``motives'' through ``motivic Galois groups'' are available (see \cite{AK} and \cite{AB} for a comprehensive account on it) but the standard conjectures remain unsolved; the fundamental difference of the approach due to Voevodsky \cite{Vh} and Ayoub \cite{Ay}  is the construction of a ``triangulated category of motives'' instead of the abelian category but the correct ``motivic t-stucture'' is missing: it implies the standard conjectures  (see \cite{Be} for the relation between the triangulated and the Tannakian approach). 

In particular, Nori's construction is modelled on singular (co)homology regarded as a representation of a diagram and it is universal with respect to (co)homology theories which are comparable with singular (co)homology: this implies that it is not available in positive characteristics, for example. The reformulation of Nori's construction by making use of syntactic categories in \cite{BCL} has not been really fruitful, so far; it is quite handy for logical purposes but the algebraic structure of a syntactic category is rather hard to manage, \eg if one would like to get a tensor structure on it. 

The unified (co)homological framework in \cite{BV} is also a first step towards the construction of ``motives'' independently of fiber functors: it was settled in the language of categorical model theory but its translation in the language of representations of quivers started parallelly with \cite{BVP}; actually,  Freyd's universal abelian category is linked with the construction of Nori motives as well as with the triangulated categories of Voevodsky motives. A tensor version of Freyd's universal abelian category shall provide tensor product of ``motives'' \eg for abelian categories modelled on a given cohomology satisfying K\"unneth formula, see \cite{BVHP} and \cite{BVPT}.

Finally, a second step towards a unified framework for (co)homology theories on categories, following \cite{BV} -- along with the consequent approach to ``motives'' -- is completely reformulated and extended in this paper by making use of the rather elementary and classical method of solving representability problems.  

\subsubsection*{Acknowledgements} 
I am happy to express my thanks to Joseph Ayoub  whose advice, interest and criticism were the indispensable ingredients for the writing of this paper. I'm also deeply grateful to Mike Prest for his constant help and precious collaboration. I also would like to dedicate this paper to the memory of V. Voevodsky, A. Grothendieck and S. Eilenberg to whom I'm intellectually tied and feel beholden. Finally, I also like to thank UZH \& FIM--ETH Z\"urich  for providing support, hospitality and excellent working conditions. 

\subsubsection*{Notations and assumptions} We shall be adopting the current conventions on small and large categories, considering a fixed universe when small or locally small is specified, \eg see \cite[I \S 0-1]{sga4}, \cite[Convention 1.4.1.]{KS} and \cite[A1.1]{El}.  For abelian categories refer to  \cite[Chap.\,1]{Gr}, \cite{GG}, \cite[Part One]{KH} and \cite[Chap.\,8]{KS}. In particular, we say that $\cS$ is a \emph{Serre subcategory} or \emph{thick} subcategory of an abelian category $\cA$ if it is closed by subobjects, quotients and extensions see  \cite[\S 1.11]{Gr} and \cite[\S 1.2]{GG}; note that in \cite[Def.\,8.3.21 iv)]{KS} the weaker version of thick is adopted. Refer to the abelian category  $\cA/\cS$ as the \emph{quotient category}, see \cite[\S 1.11]{Gr}, \cf \cite[Ex.\,8.12]{KS}. A cocomplete abelian category $\cA$ where filtered colimits are exact and possessing a generator is a \emph{Grothendieck category} (but we almost not make use of the existence of a generator); the Grothendieck quotient $\cA/\cS$ is the quotient of a Grothendieck category $\cA$ by $\cS$ thick and localizing, \ie closed under arbitrary sums, see \cite[\S 1.2]{GG}, \cf \cite[Ex.\,8.13]{KS}. 

We denote specific categories like $\Ind \cC$ the indization of  $\cC$, see \cite[I \S 8]{sga4} and \cite[Chap.\,6]{KS}, \eg  $\RMod =\Ind \Rmod$ the category of $R$-modules where $\Rmod$ is that of finitely presented $R$-modules (here $R$ shall be a commutative unitary ring unless specified) and large categories like $\Psh (\cC)$ the category of presheaves of sets and  ${\rm Cat}$ the category of (small) categories and functors. 

Denote with slanted bold letters (large) \emph{2-categories}, like $\Cat$ the 2-category of categories and functors or $\Ex$ the category of abelian categories and exact functors; we also keep in slanted bold the \emph{2-functors}: for an essential review of 2-categories and 2-functors including \emph{2-representability} refer to \cite[B1.1]{El}. 

A \emph{directed graph } or \emph{quiver} or \emph{diagram scheme} \cite[\S 1.6]{Gr} or \emph{diagram}, for short, shall be denoted by $D$. A diagram is given by a set $D_0$ of vertices and a set $D_1$ of edges, plus mappings $d_0, d_1: D_1\to D_0$; for $\alpha\in\ D_1$ we say  that $d_0(\alpha)=s\in D_0$ is the source and $d_1(\alpha)=t\in D_0$ is the target and we shall denote the edge $\alpha: s\rightarrow t$ for short. Denote  $\Graph$ the category whose objects are diagrams $D = (D_0, D_1, d_0, d_1)$ and morphisms $\varphi : D = (D_0, D_1, d_0, d_1)\to D' = (D_0', D_1', d_0', d_1')$  are pairs $\varphi = (f_0, f_1)$ of mappings $f_0 : D_0\to D_0'$ and $f_1: D_1\to D_1'$ compatible with sources and targets, \ie $f_0d_0= d_0'f_1$ and $f_0d_1= d_1'f_1$. 

\section{Universal representations}
Consider a diagram $D$ and a (small) category $\cD$. Regarding $\cD$ as a diagram recall that a morphism of directed graphs $T: D \to \cD$ is also called a representation of $D$ in $\cD$. This $T$ is called a diagram in $\cD$ from the scheme $D$ in \cite[\S 1.6]{Gr}: for each vertex $s$ we have an object $T_s \in \cD$ and for each edge $\alpha: s\rightarrow t$, a morphism $T_\alpha :T_s \rightarrow T_t$ of $\cD$.
A morphism $f:T\rightarrow S$ between representations is a collection $\{f_s\}_{s\in D_0}$ of morphisms in $\cD$ with $f_s:T_s\rightarrow S_s$ such that, for every $\alpha\in D_1$, $\alpha:s\rightarrow t$, we have
$$
\xymatrix{T_s \ar[r]^{T_\alpha} \ar[d]_{f_s} & T_t \ar[d]^{f_t} \\ S_s \ar[r]_{S_\alpha} & S_t}
$$
a commutative square in the category $\cD$. Composition of morphisms is induced by that on $\cD$ glueing such commutative squares and the identity of a representation $T$ is the collection $id_T \df \{ id_s\}:T_s\rightarrow T_s$ of identities in $\cD$. 
The category of representations is the category whose objects are representations and morphisms are morphisms of representations. 

\subsection{The representation problem} 
Let $\Rep (D, \cD)$ be the category of  representations of $D$ in the category $\cD$ (this category is denoted $\cD^D$ in \cite[\S 1.6]{Gr}) and let $F: \cD\to \cD'$ be a functor. For $T\in \Rep (D, \cD)$ we get \emph{the image of $T$ under $F$}, denoted $FT\in \Rep(D, {\cD'})$,  by setting $FT_s \df F(T_s)$ and $FT_\alpha \df  F(T_\alpha) : 
F(T_s) \rightarrow F(T_t)$ in $\cD'$. Since $F$ is a functor it preserves commutative squares and therefore if  $f =\{f_s\}: T\to S$ is a morphism of representations in $\cD$ then $Ff = \{F(f_s)\}: FT\to FS$ is a morphism of representations in $\cD'$. We get an induced functor 
 $F: \Rep (D, \cD) \to \Rep(D, {\cD'})$ and if we have a natural transformation $\eta : F \implies G : \cD\to \cD'$ we clearly obtain an induced natural transformation $\eta : F \implies G : \Rep (D, \cD) \to \Rep(D, {\cD'})$. All this being (strictly) compatible with 1-composition and 2-composition yields a (strict) 2-functor $\Rep (D, - )$ from the 2-category of categories to itself.  
  
The representation problem is the problem of 2-representing the functor of representations over a restricted domain which we call \emph{domain of representability}. More specifically, consider a 2-functor $\Phi : \Dom \to \Cat$ and consider the restriction of representations as follows
$$
 \cD\leadsto \Rep (D, \Phi\cD): \Dom \to \Cat
 $$
 and call $\Phi$-representations the objects of $\Rep (D, \Phi\cD)$. 
Denote $$\Rep_{\Phi}(D)\df \Rep (D, \Phi) $$ the resulting composite 2-functor.
\begin{defn} We say that the diagram $D$ is \emph{representable in the domain} $\Dom$  if  $\Rep_{\Phi}(D)$ is representable. If  $(\cD_D, \Delta_\Phi)$ is representing $\Rep_{\Phi}(D)$ we say that $\cD_D$ is the \emph{universal domain category} and $\Delta_\Phi : D \to \Phi\cD_D$ the \emph{$\Phi$-universal representation} (such a pair is unique up to natural equivalence).
 \end{defn}
Actually, representability in the domain $\Dom$ can be translated, as usual, with the existence of the category $\cD_D\in  \Dom $  together with $$\eta_{\cD} :\Rep (D, \Phi\cD)\cong \Dom (\cD_D, \cD)$$ an equivalence of categories which is natural as $\cD$ varies, \ie for $F : \cD \to \cD'$ in $\Dom $ we have that 
 $$
\xymatrix{\Rep (D, \Phi\cD) \ar[r]^{\simeq}_{\eta_{\cD} } \ar[d]_{\Phi F} & \Dom (\cD_D, \cD)\ar[d]^{-\c F} \\ \Rep(D, \Phi\cD') \ar[r]^{\simeq}_{\eta_{\cD' }} & \Dom (\cD_D, \cD')}
$$
is commutative.
In particular, we get the universal $\Phi$-representation $\Delta_\Phi$ such that $\eta_{\cD_D}(\Delta_\Phi) = id_{\cD_D}$ is the identity functor of  $\cD_D$ provided with the usual universal property: 
$$\xymatrix{D \ar[r]^{\Delta_\Phi} \ar[d]_T & \Phi\cD_D  \ar@{.>}[dl]^{\Phi F_T} \\  \Phi \cD}$$
saying that for any other $\Phi$-representation $T$ there is a unique (up to natural equivalence) functor $F : \cD_D \to \cD$ in $\Dom$ such that $T$ factors through the universal $\Phi$-representation.\\ 

Clearly, any diagram $D$ is representable in the domain of all categories, \ie for $\Phi = id$ the identity functor,  since  $$\Rep (D, \cD)\cong \Cat (\bar{D}, \cD)$$ where $\cD_D \df \bar{D}$ here is the path or free category of a directed graph, as it is well known.  Recall that $\bar{D}$ is the category having for objects the vertices and for morphisms the finite strings of edges. This is a functor $D\leadsto \bar{D} : \Graph \to {\rm Cat}$ which is left adjoint to the forgetful functor ${\rm Cat} \to \Graph$, where $\Graph$ is the category of directed graphs.   If we are interested to keep some properties of more sophisticated diagrams we can impose them through a quotient of the path category.  For example, some diagram schemes with commutativity conditions are considered in \cite[\S 1.6-1.7]{Gr}.
\begin{defn} 
A decorated diagram $(D, \dag )$ with an extra set of data and a set $\dag$ of commutativity relations, a \emph{$\dag$-diagram} for short, is given by a set of distinguished vertices and edges and a set $\dag$ of parallel strings of edges of the path category. Denote $\Graph^\dag$ the category of $\dag$-diagrams where the $\dag$-morphisms are morphisms of  directed graphs preserving extra data and compatible with the sets $\dag$ of edges. 
 \end{defn}
Set $$D^\dag \df \bar D/\dag$$ for the quotient of the path category forcing the commutativity relations to become actual commutative diagrams in the quotient category. Consider $\dag$-categories, \ie categories with extra data and the set $\dag$ of commutativity conditions, and $\Cat^\dag$, the 2-category of $\dag$-categories with functors with extra data and preserving the commutativity conditions. Let $\Rep^\dag (D, \cD)\subseteq \Rep (D, \cD)$ be the subcategory of  $\dag$-representations in a $\dag$-category $\cD\in \Cat^\dag$,  \ie representations with extra data which preserve data and are taking the prescribed commutativity relations to commutative diagrams.  We thus obtain an equivalence 
$$\Rep^\dag (D, \cD)\cong \Cat^\dag (D^\dag, \cD)$$ 
by construction. The induced functor $(D, \dag)\leadsto D^\dag : \Graph^\dag \to {\rm Cat}^\dag$ is left adjoint to the forgetful  functor ${\rm Cat}^\dag \to \Graph^\dag$ between the underlying categories. For example, the decoration $\dag =\otimes$ is that of   $\Cat^\otimes$ the 2-category of tensor categories and tensor functors (see \cite{DT} and \cite[Chap.\,4]{KS} for tensor category theory) and $\Graph^\otimes$ is the category of $\otimes$-diagrams (see \cite[ Def.\,2.1]{BVHP} for the notion of tensor diagram and that of tensor representation).
Note that the free category $\bar D = D^\emptyset$ is obtained for  $D$ without extra data and $\dag =\emptyset$ the empty set of commutativity conditions and $\Rep^\emptyset (D, \cD) = \Rep (D, \cD)$. 
Consider a functor $\Phi : \Dom \to \Cat^\dag$,  the restriction of representations as above 
$$ \cD\leadsto \Rep^\dag (D, \Phi\cD): \Dom \to \Cat$$ and denote $\Rep_{\Phi}^\dag(D)\df \Rep^\dag (D, \Phi) $ the resulting composite 2-functor.
\begin{defn} We say that the decorated diagram $(D,\dag )$ is \emph{representable in the domain} $\Dom$  if  $\Rep_{\Phi}^\dag(D)$ is representable. If  $(\cD_D^\dag, \Delta_\Phi^\dag)$ is representing $\Rep_{\Phi}^\dag(D)$ we say that $\Delta_\Phi ^\dag: D \to \Phi\cD_D^\dag$ is the \emph{$\Phi$-universal $\dag$-representation}.
 \end{defn}
For such a decorated diagram $(D, \dag )$ denote $$\delta^\dag : D\to D^\dag$$ the image representation of the immersion $D\into \bar{D}$ under $\bar D\onto D^\dag$ the quotient functor. We have:
\begin{lemma} \label{domadj}
If  $\Phi :\Dom \to \Cat^\dag$ has a left adjoint $\Psi$ then $\cD_D^\dag\df \Psi (D^\dag)$ and
$$\xymatrix{D\ar[r]_{\delta^\dag} \ar@/^1.7pc/[rr]^{\Delta_\Phi^\dag} & D^\dag \ar[r]_{\xi_{D^\dag}\hspace{1.2cm}} & \Phi  \Psi (D^\dag) = \Phi\cD_D^\dag}$$
provided by the image of  $\delta^\dag$ under the unit $\xi_{D^\dag}$ of the adjunction, yields a $\Phi$-universal $\dag$-representation in the domain $\Dom$. 
\end{lemma} 
\begin{proof} In fact, the diagram $D$ is representable in the domain $\Dom$  if and only there is a natural equivalence
$$\Cat^\dag (D^\dag, \Phi\cD) \cong \Dom (\cD_D^\dag , \cD)$$ 
for $\cD\in \Dom$. If we have the adjoint  $\Psi : \Cat^\dag\to \Dom$ this can be fulfilled by defining $\cD_D^\dag\df \Psi (D^\dag)$. 
\end{proof}
Note that in the situation of Lemma \ref{domadj} we could have $\Cat^\dag = \Cat$ for $\dag\neq \emptyset$.
\begin{defn} \label{catdiag}
 Call $(D, \dag )$ a \emph{categorical diagram} if it is a $\dag$-diagram such that  $\Rep^\dag (D, - )\cong \Cat (D^\dag , - )$.
\end{defn}
 We then have:
 \begin{lemma} Let $(D, \dag )$ be a categorical diagram and $\Phi :\Dom \to \Cat$ with a left adjoint $\Psi$. Then 
$$\cD_D\df \Psi (\bar D)\onto \cD_D^\dag\df \Psi (D^\dag)$$ is a quotient. 
\end{lemma} 
\begin{proof}
 In fact, $\Psi$ being a left adjoint, it preserves quotients.
\end{proof}
\begin{examples}\label{decex}
a) If  $D=\cC$ is the underlying directed graph of a category $\cC$ then let $(D, \circ)$ be the decorated diagram ``with identities''  and with all commutativity relations given by the composition in the category $\cC$. Then a $ \circ$-representation of $D$ in $\cD$ is a functor from $\cC$ to $\cD$. Therefore $D^\circ \cong \cC$ given by the quotient of $\bar{D}$ obtained by imposing all relations to become commutative diagrams and where $id _C = (e_v)$  is the empty path for every vertex $v=C\in \cC$. This $(D, \circ)$  is the prototype of a categorical diagram. 

b) Let $(D, \otimes)$ be a $\otimes$-diagram (see \cite[Def.\,2.1]{BVHP}). A $\otimes$-representation of $D$ in a tensor category $(\cD, \otimes)$  (see \cite[Def.\,2.7]{BVHP}) corresponds to a $\otimes$-functor $(D^\otimes, \otimes) \to (\cD, \otimes)$ and $\Rep^\otimes (D, - )\cong \Cat^\otimes (D^\otimes , - )$ where, clearly, $\Cat^\otimes\neq \Cat$. This  $(D, \otimes)$ is not a categorical diagram. 
\end{examples}
\subsection{Additive representations}
We are interested in forgetful functors $\Phi$ on domain categories $\Dom$ that factor through $\Add_R$ the category of $R$-linear additive categories and $R$-linear additive functors. These domain categories shall be determined by extra structures or properties of additive categories and functors $F : \cD\to \cD'$ in $\Dom$ shall be additive functors $\Phi F : \Phi \cD \to \Phi\cD'$ with extras and preserving these structures or properties. Actually, we can always assume such a factorisation, \cf  \cite[Lemma 1.4]{BVP}.
\begin{propose}\label{repadd}
 Any diagram is representable in the domain of ($R$-linear) additive (resp.\/ pseudo-abelian) categories and ($R$-linear) additive functors. \end{propose}
\begin{proof} We just apply Lemma \ref{domadj} as for $\Dom = \Add_R$ the forgetful functor $\Phi: \Add_R \to \Cat$ is provided with a left adjoint $\Psi$. It is well known that we can form the pre-additive $R$-linear enrichment $R\cC$ and the additive envelope $R\cC^+$ of any category $\cC$, so that $\Psi (\cC) \df R\cC^+$. Applying this to the path category $\cC= \bar{D}$ of a diagram $D$ we obtain  $$\Delta^+: D\to R\bar{D}^+$$ the universal $R$-linear additive representation. Finally, we can form the pseudo-abelian completion of an additive category and again applying Lemma \ref{domadj} we are done.
\end{proof} 
\begin{example}\label{pointR}
For $D=\{*\}$ the singleton one vertex diagram we get $\bar{D} ={\bf 1}$ the one point category, $R{\bf 1}=R$ is the ring as a preadditive $R$-linear category and $R\bar{D}^+=R^+$ the additive completion of $R$. The category $R^+$ is equivalent to the full subcategory of $\Rmod$ given by finitely generated free $R$-modules. The representation $\Delta^+: \{*\}\to R^+$ sends $\ast$ to $R$. 
\end{example}
\begin{remark}
Recall that, even in the non additive case, we always can represent any diagram $D$ in the domain of the  categories where every idempotent is a split idempotent ${\sf Idem}$ (also called idempotent complete or Cauchy complete or Karoubian, see \cite[A1.1]{El}). As ${\sf Idem} \to \Cat$ has a left adjoint we can apply the same argument as in Proposition \ref{repadd}. In fact, for any category $\cC$ we can always form the idempotent completion $\cC^{\rm idem}$ in such a way that $\Psh (\cC)=\Psh (\cC^{\rm idem})$ and we can recover $\cC^{\rm idem}$ as those presheaves which are retracts of representables. Actually, $\cC$ can be recovered as $(\Ind \cC)^{\rm fp}$, \ie as the compact or finite presentation objects of $\Ind \cC$, see \cite[Def.\,6.3.3]{KS}, if and only if $\cC$ is idempotent
complete, see \cite[Ex.\,6.1]{KS}.
\end{remark}

\subsection{Abelian representations}
A key fact that we recall here is Freyd's free abelian category constructed out of any additive category, see Freyd \cite[\S 4.1]{Fr} and Prest \cite[Thm.\,4.3]{P}. See also Krause \cite[\S 11-12]{KH}. We consider its $R$-linear variant following \cite[\S 1]{BVP}.  Let $\cR$ be an $R$-linear additive category. Denote $\cR\mbox{-}{\rm Mod} \df \Add_R (\cR,R\mbox{-}{\rm Mod})$ (as usual, see \cite[Chap.\,2]{P}) the big category of $R$-linear additive functors regarded as $\cR$-modules along with the contravariant Yoneda embedding $\cR^{op}\into \cR\mbox{-}{\rm Mod}$. We may look at ``internal $\cR$-modules'' defined as 
$$\cR\mbox{-}{\rm Mod}(\cA) := \Add_R(\cR, \cA)$$ where $\cA$ is an (essentially small) $R$-linear additive category.   Actually, any object of an additive category $\cA$ is a commutative group object in $\cA$ and for $\cR = R$ (by abuse of notation, it should be $R^+$, \cf Example \ref{pointR}) we get $\text{$R$-Mod }(\cA) =\cA$, \ie  $\Add_R(R^+, \cA)\cong \cA$, indeed. By the way $$\cR\mbox{-}{\rm Mod}(-) : \Dom \to \Cat$$ is a canonical 2-functor to ask for representability over a suitable domain ${\sf Dom}$. We also have the full subcategory $\cR\mbox{-}{\rm mod}\df \cR\mbox{-}{\rm Mod}^{\fp}$ determined by finite presentation $\cR$-modules, \ie cokernels of representables, (see \cite[Chap.\,1]{P}) and note that Yoneda yields $\cR \into \cR\mbox{-}{\rm mod}^{op}$.  Let $\Lex_R$ be the 2-category of (essentially small) $R$-linear left exact additive categories and $R$-linear left exact functors.
\begin{lemma}\label{lexrmod}
The 2-functor  $\cR\mbox{-}{\rm Mod}(-)$ is representable  by $\cR \into \cR\mbox{-}{\rm mod}^{op}$ on  $\Dom =\Lex_R$, \ie there is an equivalence $\cR\mbox{-}{\rm Mod}(\cA) \cong {\sf Lex}_R (\cR\mbox{-}{\rm mod}^{\rm op}, \cA)$ which is
natural in $\cA$, $\cA \in \Lex_R$.
\end{lemma}
\begin{proof} In fact, we have that $\cR\mbox{-}{\rm mod}$ is right exact, i.e. it has cokernels, and any additive functor from $\cR$ to $\cA$ can be lifted (uniquely-up-to-equivalence) to a left exact functor from $\cR\mbox{-}{\rm mod}^{\rm op}$ to $\cA$ as it easily follows from the arguments in \cite[\S 1]{BVP} or the proof of  \cite[Thm.\,4.3]{P}.
\end{proof}
Thus $\cR \into \cR\mbox{-}{\rm mod}^{op}$ is the universal $\cR$-module in $R$-linear left exact additive categories. Applying the previous Lemma twice one obtains the following, see  \cite[\S 1]{BVP} or \cite[Thm.\,4.1, Cor.\,4.2 \& 4.3]{P}.
\begin{thm}[Universal $\cR$-module]\label{Freyd}
There is an abelian category $\Ab_R(\cR)$ and an additive functor $\cR\to \Ab_R(\cR)$  representing the 2-functor $\cR\mbox{-}{\rm Mod}(-)$ on  $\Dom =\Ex_R$ the 2-category of (essentially small) $R$-linear abelian categories and $R$-linear exact functors. Moreover, the $R$-linear additive fully faithful functor $|\cR|$ induced by the Yoneda embeddings
$$|\cR|: \cR\into \cR\mbox{-}{\rm mod}\mbox{-}{\rm mod} = \Add_R (\Add_R (\cR,R\mbox{-}{\rm Mod})^{\rm fp}, R\mbox{-}{\rm Mod})^{\rm fp}\cong \Ab_R(\cR)$$
is representing $\cR\mbox{-}{\rm Mod}(-)$. 
\end{thm}  
For any category $\cC$, considering the $R$-linear additive completion $R\cC^+$ of $\cC$, as explained in the proof of Proposition \ref{repadd}, we see that $\Ab_R(R\cC^+)$ is Freyd's universal abelian category of $\cC$. 
\begin{defn} \label{hieraticdef} Let $\cR=R^+$ be the additive category generated by the point category and $R\mbox{-}{\rm mod} \cong \Add_R (R^+,R\mbox{-}{\rm Mod})^{\fp}$ the finitely presented $R$-modules. 
We shall refer to a \emph{hieratic $R$-module} for an object of $\iRMod\df\Add_R (R\mbox{-}{\rm mod}, R\mbox{-}{\rm Mod})$ and \emph{finitely presented hieratic $R$-module} for an object of 
$$\iRmod\df \iRMod^{\fp} = \Add_R (R\mbox{-}{\rm mod}, R\mbox{-}{\rm Mod})^{\fp}\cong \Ab_R(R^+)$$
(this convention is also providing $\iRmod = R^+\mbox{-}{\rm mod}\mbox{-}{\rm mod}$ simplifying the notation).
\end{defn}
The Theorem \ref{Freyd} means that the double Yoneda $|\cR| \in\cR\mbox{-}{\rm Mod}(\Ab_R(\cR))$ is the universal $\cR$-module in abelian categories: every $\cR$-module $M\in \cR\mbox{-}{\rm Mod}(\cA)$ is the image of $|\cR|$ under an exact functor $F_M: \Ab_R(\cR)\to \cA$. 
Note that if  $\cA$ is an abelian subcategory of $\Ab_R(\cR)$ which contains the image of $|\cR|$ then the inclusion of $\cA$ in $\Ab_R(\cR)$ is an equivalence (see \cite[Lemma 4.12]{P}). Actually, equivalently,  we see that the canonical  forgetful functor $$\Phi_R : \Ex_R \to \Add_R$$ has a left adjoint $\Psi_R: \Add_R\to \Ex_R$, where $\Psi_R (\cR)\df \Ab_R(\cR)$ is Freyd's $R$-linear abelian category. 
  We then have the following, \cf \cite{BVP}:
\begin{thm}[Universal abelian representation] \label{unrep}
Any diagram is representable in the domain of $R$-linear abelian categories.\end{thm}
\begin{proof} We see that $\cD_D\df \Ab_R (R\bar{D}^+)\in \Dom =  \Ex_R$ is the universal abelian category by Theorem \ref{Freyd}, the proof of Proposition \ref{repadd} and Lemma \ref{domadj}.\end{proof}
We obtain the universal abelian representation $\Delta$ of  $D$ in $\Ab_R (R\bar{D}^+)$ as the image of $\Delta^+$ (in the proof of Proposition \ref{repadd}) under $R\bar{D}^+\to \Ab_R (R\bar{D}^+)$ the canonical $R$-linear embedding. 
\begin{defn}\label{unrepdef}
For a diagram $D$ we shall denote 
$$\Delta : D \to  \Ab_R(D)\df \Ab_R(R\bar{D}^+)$$ the \emph{universal $R$-linear abelian representation}. For $R=\Z$ we simply write $\Ab (D)$ omitting reference to $R$. For $D =\{*\}$ the singleton one vertex diagram we shall denote 
$$\Ab_R\df \Ab_R(\{*\})\cong \iRmod$$ identified with the category of finitely presented hieratic $R$-modules, see Definition \ref{hieraticdef}, \ie Freyd's free $R$-linear abelian category of the ring $R$. Let $|R|$ be the \emph{universal hieratic $R$-module} corresponding to $\Delta (*)$  and given by the image of $R$ under the double Yoneda $R^+\to \iRmod$. 
\end{defn} 
Note that $\Ab_R(D)$ has the same size of $D$ by construction. For any $R$-linear abelian category $\cA$  we then have a natural equivalence
$$T\leadsto F_T : \Rep (D, \cA)\cong \Ex_R (\Ab_R(D), \cA)$$
where $\Delta$ corresponds to the identity, \ie $F_\Delta = id$. In particular, for $D =\{*\}$, any  $A\in \cA=\text{$R$-Mod }(\cA)$ yields a unique (non necessarily faithful!) $F_A: \Ab_R\to \cA$ such that $F_A(|R|) = A$.  Under the equivalence $\Ab_R\cong \iRmod$, mentioned in Definition \ref{unrepdef}, we shall denote $$r_A: \iRmod\to \cA$$ the $R$-linear exact functor $F_A$ induced by an object $A\in \cA$. 

We also have that $\Ab_R(-)$ is functorial on diagrams and the previous equivalence is functorial in both variables.

Note that we have two canonical associated abelian categories along with exact 
embeddings
$${\rm Sim}^\oplus \Ab_R(D)\into \Ab_R(D)\into \Ind \Ab_R(D)$$
where:
\begin{itemize}
\item[-] ${\rm Sim}^\oplus \Ab_R(D)$ is the abelian full subcategory of semisimple objects of $ \Ab_R(D)$ which is also the thick subcategory 
generated by the simple objects ${\rm Sim} \Ab_R(D)$ (\cf \cite[Def.\,8.3.16-8.3.21 ]{KS}) and
\item[-] $\Ind \Ab_R(D)$ is the indization of $\Ab_R(D)$ which is a Grothendieck $R$-linear category (\cf \cite[Def.\,8.3.24 \& Thm.\,8.6.5]{KS}). 
 \end{itemize}
\begin{propose}\label{indadjoint}
For an exact functor $F : \cA \to \cB$ from an (essentially small) abelian category $\cA$ to a Grothendieck category $\cB$ there exists $\Ind \cA\to \cB$ a unique (up to unique isomorphism) extension   along $\cA\into \Ind \cA$ which is exact and preserving (small) filtered colimits \ie $\cA\into \Ind \cA$ is 2-universal in Grothendieck categories.  
\end{propose}
\begin{proof} Well known: a proof can be extracted from \cite[I Prop.\,8.7.3 \& \S 8.9]{sga4} or \cite[Prop.\,6.1.9, Prop.\,6.3.1, Cor.\,6.3.2, Prop.\,8.6.6 \& Cor.\,8.6.8]{KS}. 
In fact,  if  $F : \cA \to \cB$ is exact then $\Ind F : \Ind \cA \to \Ind \cB$ is exact and if $\cB$ is a cocomplete abelian category the exact embedding $\cB \into \Ind \cB$ has a left adjoint/inverse 
$$L: \limdir{ }\leadsto \limd{ } : \Ind \cB\to \cB$$ 
in such a way that, firstly,  we obtain $F$ as the composition of $\cA\into \Ind \cA$ with $L\circ \Ind F : \Ind \cA\to \cB$ and, secondly, this latter $L\circ \Ind F$ commutes with filtered colimits and is exact, since in $\cB$ the filtered colimits are exact.
\end{proof}
Note that $\Ind \cR\mbox{-}{\rm mod} \cong \cR\mbox{-}{\rm Mod}$ (by Lemma \ref{lexrmod} see also \cite[Thm.\,11.1.15]{KH} and \cite[Cor.\,3.3 \& 3.8]{P}) and 
$\Ind \iRmod\cong \iRMod$ is the Grothendieck category of hieratic $R$-modules. 
\begin{cor} \label{indrep}
For any $R$-linear Grothendieck category $\cA$  we then have a natural equivalence between 
$ \Rep (D, \cA)$ and the category of exact functors preserving filtered colimits from 
$\Ind \Ab_R(D)$ to $\cA$; the image of $\Delta$ under $\Ab_R(D)\into \Ind \Ab_R(D)$ yields 
$$\Ind \Delta : D \to  \Ind \Ab_R(D)$$
the universal Grothendieck representation. 
\end{cor}
\begin{remark}
For a site $(\cC, J)$ let $\Shv_J(\cC)$ be the Grothendieck topos of $J$-sheaves of sets. We have a functor, via Yoneda, from $\cC$ to $\Shv_J(\cC)$ that can be enriched to a functor to $\Shv_J(\cC; R)$ the category of $J$-sheaves of $R$-modules on $\cC$. This latter is a Grothendieck category.  Actually, we associate   $X\in \cC$ with $R(X)$ the $J$-sheaf associated to the $R$-free module generated by the representable presheaf of sets (in Voevodsky notation \cite[Prop. 2.1.1 \& \S 3.3]{Vh}). Therefore, we always obtain an $R$-linear exact functor  $$\Ind \Ab_R(R\cC^+)\to \Shv_J(\cC;R)$$ where 
$\Ind \Ab_R(R\cC^+)$ is Freyd's universal Grothendieck category of $\cC$. 
\end{remark} 
The dual representation $$\Delta^{op}: D^{op} \to  \Ab_R(D)^{op}$$ is the universal representation of the dual diagram. We have:
\begin{thm}[Abstract duality] \label{dualrep}
For a diagram $D$ and the dual diagram $D^{op}$ we have that  $\Ab_R(D^{op})\cong\Ab_R(D)^{op}$ and $\Pro \Ab_R(D)
\cong (\Ind \Ab_R(D^{op}) )^{op}$.
\end{thm}
\begin{proof} Using Theorem \ref{unrep}  we have that 
$$\Rep (D^{op}, \cA)\cong \Rep (D, \cA^{op})\cong \Ex_R (\Ab_R(D), \cA^{op})\cong \Ex_R (\Ab_R(D)^{op}, \cA)$$ 
proving the first claim. The second claim follows from the first.
\end{proof}
Now consider an invertible morphism $D'\to D^{op}$ of diagrams, \eg $D' =D$ with an anti-involution $e:D\to D^{op}$.  
\begin{cor}[Concrete duality] \label{dualinv}
If $D'\to D^{op}$ is an invertible morphism of diagrams then $\Ab_R(D')\by{\simeq}\Ab_R(D)^{op}$ are equivalent abelian categories.
\end{cor}
\begin{proof}
The invertible morphism $D'\to D^{op}$ yields an R-linear exact functor $\Ab_R(D')\to \Ab_R(D^{op})$ which is an equivalence by the functoriality of $\Ab_R (-)$. Actually, by Theorem \ref{dualrep}, we have that $\Ab_R(D^{op})\cong \Ab_R(D)^{op}$. Therefore, by composition, we obtain an equivalence $\Ab_R(D')\by{\simeq} \Ab_R(D)^{op}$ as claimed. 
\end{proof}
\begin{defn}\label{dagdef}
For $(D, \dag)$ a $\dag$-diagram define the \emph{$\dag$-decoration $\Ab_R (D)^\dag$} of  $\Ab_R (D)$ to be the quotient 
$$\pi^\dag : \Ab_R (D)\onto \Ab_R (D)^\dag$$
where $\ker \pi^\dag$ is the thick subcategory generated by the images $\im (\Delta_\alpha - \Delta_\beta)$ as $\alpha$ and $\beta$ varies among all parallel edges providing commutativity of $\dag$-relations.  
\end{defn}
\begin{lemma}\label{dag}If  $(D, \dag)$ is a categorical diagram, see Definition \ref{catdiag}, then the induced $\dag$-representation
$\Delta^\dag : D \to  \Ab_R(D)^\dag$
is the universal abelian $\dag$-representation. 
\end{lemma}
\begin{proof}
The category $D^\dag$ is a quotient of the path category $\bar D$ and we have an exact functor $\pi^\dag : \Ab_R(R\bar{D}^+)\to \Ab_R(RD^{\dag, +})$. The equivalence classes $[\alpha]=[\beta]$ for parallel edges in $\dag$-relations yield $ \Delta_\alpha = \Delta_\beta$ in $ \Ab_R(RD^{\dag, +})$ so that $\ker \pi^\dag$ should contain all such 
$\im (\Delta_\alpha - \Delta_\beta)$. Thus the thick subcategory $\ll \im (\Delta_\alpha - \Delta_\beta)\gg$ is contained in $\ker \pi^\dag$.   Since $\Rep^\dag (D, - )\cong \Cat (D^\dag , - )$ we have that  a $\dag$-category is just a category so that $  \Ab_R (D)^\dag\df \Ab_R (D)/\ll \im (\Delta_\alpha - \Delta_\beta)\gg$ is universal with respect to $\dag$-representations and $\Delta^\dag : D \to  \Ab_R(D)^\dag\cong \Ab_R(RD^{\dag, +})$ is the universal abelian $\dag$-representation.
\end{proof}
We have the following important subcategory (see \cite[Def.\,1.1]{BVHP}):
\begin{defn}\label{flatAb}
Denote $\Ab_R(\cR)^\flat$ the smallest full subcategory containing the objects in the image of $|\cR|:\cR\to \Ab_R(\cR)$ and closed under kernels.
\end{defn} Note that, moreover,  this $\Ab_R(\cR)^\flat$ is precisely the image of $\cR\mbox{-}{\rm mod}^{op}$ under the Yoneda embedding into $\Ab_R (\cR)$ (see \cite[Remark 1.2]{BVHP}). All objects of $\Ab_R(\cR)^\flat$ are projectives in $\Ab_R(\cR)$.
\begin{defn}\label{proj}
Say that a $R$-linear left exact subcategory $\cP$ of an $R$-linear abelian category $\cA$ is a \emph{$\mathfrak{p}$-subcategory} if the objects of $\cP$ are projective objects of $\cA$. Say that an $\cR$-module $M\in  \cR\mbox{-}{\rm Mod}(\cA)$ is \emph{projective} if $M: \cR\to \cP\into \cA$ factors through a $\mathfrak{p}$-subcategory. Let $\cR\mbox{-}{\rm Proj}(\cA)\subset \cR\mbox{-}{\rm Mod}(\cA)$ be the full subcategory of projective $\cR$-modules in $\cA$. 
\end{defn}
Consider the 2-category $\Ex^{\mathfrak{p}}_R$ whose objects are pairs $(\cA, \cP)$ where $\cP$ is a $\mathfrak{p}$-subcategory of $\cA$ and morphisms  $(\cA, \cP)\to  (\cA', \cP')$ are $R$-linear exact functors $\cA\to \cA'$ which restrict to a functor $\cP\to \cP'$.  Given $(\cA, \cP)$ let 
$\cR\mbox{-}{\rm Proj}(\cA, \cP)\subset \cR\mbox{-}{\rm Proj}(\cA)$ given by all $\cR$-projective modules through $\cP$. If   $(\cA, \cP)\to  (\cA', \cP')$ then $\cR\mbox{-}{\rm Proj}(\cA, \cP)\to \cR\mbox{-}{\rm Proj}(\cA', \cP')$ in such a way that 
$$\cR\mbox{-}{\rm Proj} : \Ex^{\mathfrak{p}}_R\to \Cat$$
is a 2-functor
\begin{propose} \label{penvelope}
Let $\cR$ be an $R$-linear additive category. 
The 2-functor $\cR\mbox{-}{\rm Proj}$ is representable by the following 
 $|\cR| : \cR\to \Ab_R(\cR)^\mathfrak{p}\subset \Ab_R(\cR)$
 universal projective $\cR$-module. 
\end{propose}
\begin{proof} The category  $\Ab_R(\cR)^\flat\subset \Ab_R(\cR)$ is a $\mathfrak{p}$-subcategory and therefore $|\cR|$ is a projective module by construction. The pair  $(\Ab_R(\cR), \Ab_R(\cR)^\mathfrak{p})$ is clearly universal by Theorem \ref{Freyd} and Lemma \ref{lexrmod}. \end{proof}
 In particular, the category $$\Ab_R(D)^\flat \df \Ab_R(R\bar{D}^+)^\flat\subset \Ab_R(D)$$  is a $\mathfrak{p}$-subcategory.

\begin{defn}
A \emph{$\mathfrak{p}$-representation} is a representation of a  diagram $D$ in  an $R$-linear abelian category which factors through a 
$\mathfrak{p}$-subcategory.
\end{defn}
Let $\Dom = \Ex^{\mathfrak{p}}_R$ be  the previous category: we call it the domain of $\mathfrak{p}$-subcategories of $R$-linear abelian categories. Denote $$\Rep^\mathfrak{p} (D): \Ex^{\mathfrak{p}}_R\to \Cat$$ the 2-functor of $\mathfrak{p}$-representations. We then get the \emph{universal abelian $\mathfrak{p}$-representation} or  \emph{$\mathfrak{p}$-envelope} of a diagram.We have the following:
\begin{thm}[Universal abelian $\mathfrak{p}$-representation] \label{penv}
Any $D$ is representable in the domain of $\mathfrak{p}$-subcategories of $R$-linear abelian categories, \ie $\Rep^\mathfrak{p} (D)$ is representable by a universal $\mathfrak{p}$-representation $$\Delta^\mathfrak{p}: D \to \Ab_R (D)^{\flat}\subset \Ab_R (D)$$ 
\end{thm}
\begin{proof}  The $\mathfrak{p}$-representation $\Delta=\Delta^\mathfrak{p}$ is universal with respect to 
$\mathfrak{p}$-representations as it easily follows from Theorem \ref{unrep} and Proposition \ref{penvelope}. 
\end{proof}

\subsection{Tensor representations}
A straightforward tensor enrichment of the previous construction is that obtained in \cite{BVHP} for $\dag =\otimes$ the $\otimes$-decoration, in our current terminology. For an $R$-linear abelian $\otimes$-category $(\cA, \otimes)$ we mean  an $R$-linear abelian category $\cA$ endowed with a right exact tensor product $\otimes: \cA\times \cA\to\cA$. Recall \cite[Def.\,1.9]{BVHP} where $\flat$-subcategories have been introduced:
\begin{defn}\label{flatdef}
A  \emph{$\flat$-subcategory} $\cA^\flat\subset \cA$ is  a full $R$-linear additive subcategory of flat objects which is closed under kernels. If $(\cR, \otimes)$ is an $R$-linear additive tensor category a \emph{tensor  $\cR$-module} $M\in  \cR\mbox{-}{\rm Mod}^\otimes(\cA)\subset \cR\mbox{-}{\rm Mod}(\cA)$ is an additive $R$-linear tensor functor $M: \cR\to \cA$. Say that $M$ is \emph{flat} if $M: \cR\to \cA^\flat\into \cA$ factors through a $\flat$-subcategory. 
 Let $\cR\mbox{-}{\rm Flat}^\otimes (\cA)\subset \cR\mbox{-}{\rm Mod}^\otimes(\cA)$ be the full subcategory of  flat $\cR$-modules in $\cA$.
\end{defn}
Actually, consider the 2-category $\Ex^{\otimes, \flat}_R$ whose objects are triples $(\cA, \cA^\flat, \otimes)$ where $\cA$ is an abelian $\otimes$-category and  $\cA^\flat$ is a $\flat$-subcategory and morphisms  $(\cA, \cA^\flat, \otimes)\to  (\cB, \cB^\flat, \otimes)$ are $R$-linear exact $\otimes$-functors $\cA\to \cB$ which send $\flat$-subcategories to $\flat$-subcategories
We easily obtain a 2-functor
$$\cR\mbox{-}{\rm Flat}^\otimes :\Ex^{\otimes, \flat}_R\to \Cat$$

Considering $\Ab_R(\cR)^\flat$ in Definition \ref{flatAb} we obtain the universal flat $\cR$-module. The following fact is a reformulation of \cite[Prop.\,1.8 \& Prop.\,1.10]{BVHP}. 
\begin{propose} \label{envelope}
Let $(\cR, \otimes)$ be an $R$-linear additive tensor category. 
The 2-functor $\cR\mbox{-}{\rm Flat}^\otimes$ is representable by the following 
 $|\cR| : \cR\to \Ab_R(\cR)^\flat\subset \Ab_R(\cR)$ flat $\cR$-module. 
\end{propose}
\begin{proof} We have to show that $|\cR|\in  \cR\mbox{-}{\rm Flat}^\otimes(\Ab_R(\cR), \Ab_R(\cR)^\flat, \otimes)$ is universal. The category $(\Ab_R(\cR), \otimes)$ is endowed with a canonical tensor structure such that the tensor product $\otimes$ is right exact, 
$|\cR|: \cR\to \Ab_R(\cR)$ is a $\otimes$-functor and all objects in $\Ab_R(\cR)^\flat \subset \Ab_R(\cR)$ are flat, so that  $\Ab_R(\cR)^\flat$ is a $\flat$-subcategory, see \cite[Prop.\,1.8]{BVHP} for details: this implies that the triple  $(\Ab_R(\cR), \Ab_R(\cR)^\flat, \otimes)\in \Ex^{\otimes, \flat}_R$ and that  $|\cR|$ is a flat $\cR$-module. To see the universal property,  for a flat $\cR$-module $M\in \cR\mbox{-}{\rm Flat}^\otimes (\cA, \cA^\flat, \otimes)$ a $\otimes$-functor $M: \cR\to \cA$ which factors through $\cA^\flat$ is given and the induced unique exact functor $\Ab_R (\cR)\to \cA$  is a $\otimes$-functor sending $\Ab_R(\cR)^\flat$ to $\cA^\flat$ compatibly with $M$ by Theorem \ref{Freyd} and Lemma \ref{lexrmod}, \cf \cite[Prop.\,1.10]{BVHP}. 
\end{proof}

 For the decoration  $\dag = \otimes$ we may consider $(D,\otimes )$ a set of extra data and conditions on a diagram, which we called a $\otimes$-diagram in \cite[Def.\,2.1]{BVHP}, in such a way that $D^\otimes$ is actually a tensor category.  We have that $(D^\otimes , \otimes)$ is the universal tensor category associated to a $\otimes$-diagram $(D, \otimes)$ providing an equivalence $$\Rep^\otimes (D, \cD)\cong \Cat^\otimes(D^\otimes, \cD)$$  for any $(\cD, \otimes)$ tensor category where $\Rep^\otimes (D, \cD)$ is the category of  $\otimes$-representations, in the sense of \cite[Def.\,2.7]{BVHP}. This is \cite[Lemma 2.3 \& Prop.\,2.8 (1)]{BVHP}. 
We let $\Dom = \Ex^{\otimes, \flat}_R$ be  the previous category: we call it the domain of $\flat$-subcategories of $R$-linear abelian tensor categories. 
\begin{defn}
A \emph{$\flat$-representation} is a $\otimes$-representation of a  $\otimes$-diagram $(D,\otimes )$ in  an $R$-linear abelian tensor category which factors through a $\flat$-subcategory.
\end{defn}
Denote $$\Rep^\flat (D): \Ex^{\otimes, \flat}_R\to \Cat$$ the 2-functor of $\flat$-representations. We then get the \emph{universal abelian $\flat$-representation} or  \emph{$\flat$-envelope} of a $\otimes$-diagram: this is a $\otimes$-representation in an abelian $\otimes$-category, which factors through a $\flat$-subcategory, universally. 
We have the following, \cf  \cite[Thm.\,2.9]{BVHP}:
\begin{thm}[Universal abelian $\flat$-representation] \label{uniflat}
Any $\otimes$-diagram $(D, \otimes)$ is representable in the domain of $\flat$-subcategories of $R$-linear abelian tensor categories, \ie $\Rep^\flat (D)$ is representable by a universal $\flat$-representation $$\Delta^\flat: D \to (\Ab_R (D)^\otimes)^\flat \subset \Ab_R (D)^\otimes$$
where $\Ab_R (D)^\otimes$ is the $\otimes$-decoration of $\Ab_R (D)$, see Definition \ref{dagdef}. 
\end{thm}
\begin{proof}  We can make an $R$-linear additive enrichment $\cR \df RD^{\otimes, +}$ of $D^\otimes$ with the bilinear extension of $\otimes$ providing a $\otimes$-structure $(\cR, \otimes)$.  This is \cite[Prop.\,2.5 \& Prop.\,1.8 (2)]{BVHP}. We have that $\Ab_R (D)^\otimes\cong \Ab_R (\cR)$ as abelian categories and $(\Ab_R (D)^\otimes, \otimes )$  is endowed with the $\otimes$-structure given by  \cite[Prop.\,1.8]{BVHP}. We also have the $\flat$-subcategory  $(\Ab_R (D)^\otimes)^\flat\df \Ab_R (\cR)^\flat$. The induced $\flat$-representation $\Delta^\flat$ is universal with respect to $\flat$-representations as it easily follows from Proposition \ref{envelope}.
\end{proof}

\section{Universal homologies} 
Let $\cC$ be a category and let $\cC^\square$ be the category whose objects are morphisms of a distinguished subcategory of $\cC$ and morphisms are commutative squares of $\cC$. We shall investigate (co)homology theories on $\cC$ and $\cC^\square$ showing the existence of universal theories: these are obtained  by applying the universal representation results of the previous section.

\subsection{Homology and cohomology theory}
Consider $H =\{ H_i\}_{i\in \Z}$ a $\Z$-indexed family of functors $H_i : \cC\to \cA$ where $\cA$ is an $R$-linear abelian category. A morphism  $\varphi : H \to H'$ is a collection $\varphi =\{\varphi_i \}_{i\in \Z}$ of natural transformations $\varphi_i :H_i \to H_i'$. We also have the contravariant analogue where $H =\{ H^i\}_{i\in \Z}$ is a $\Z$-indexed family of functors $H^i : \cC^{op}\to \cA$. 
\begin{defn} \label{homology}Call such a family  $H =\{ H_i\}_{i\in \Z}$ of functors a \emph{homology} on $\cC$ with values in $\cA$ and the contravariant version  $H =\{ H^i\}_{i\in \Z}$ a \emph{cohomology} on $\cC$ with values in $\cA$. We shall denote by $\Hom (\cC, \cA)$ the category of homologies on $\cC$ with values in $\cA$.  The category of cohomologies shall be denoted $\Coh (\cC, \cA)$. For $\{ H_i\}_{i\in \Z}:\cC\to \cA$ define $H^{op}$ by $\{ H^i\df H_{i}\}_{i\in \Z}:\cC^{op}\to \cA^{op}$ the \emph{opposite} cohomology. 
\end{defn}
Note the equivalence $\Hom (\cC, \cA^{op})\cong \Coh (\cC, \cA)$ between homologies with value in $\cA^{op}$ and cohomologies with values in $\cA$. 
For an $R$-linear exact functor $F:\cA \to \cA'$ and $H\in \Hom (\cC, \cA)$ we obtain $FH = \{FH_i\}_{i\in \Z} \in \Hom (\cC, \cA')$ by composition. This is making up a 2-functor 
$$
\cA \leadsto  \Hom (\cC, \cA): \Ex_R \to \Cat
$$
from the 2-category $\Ex_R$ of $R$-linear abelian categories and exact functors to categories.   Note that for any object $A\in \cA$ we have the constant homology
$X\leadsto H_i (X)=A$ providing an embedding $\cA\into \Hom (\cC, \cA)$.  Therefore, the forgetful functor $ \Ex_R \to \Cat$ can be regarded as a subfunctor of  $\Hom (\cC, - )$  determined by constant homologies. A key result is the following. 
\begin{thm}[Universal homology] \label{pure}
For any category $\cC$ the functor $\Hom (\cC, - )$ is representable, \ie  there is a universal homology $$H = \{ H_i\}_{i\in \Z} :\cC\to \cA(\cC)$$
with values in $\cA(\cC)$ a $R$-linear abelian category. If $F:\cC\to \cD$ is a functor, $H^\cC$ and $H^\cD$ are the universal homologies in $\cA(\cC)$ and $\cA(\cD)$, respectively, then there is a unique $R$-linear exact functor $r_F : \cA(\cC)\to \cA(\cD)$ such that $r_FH^\cC =H^\cD F$.
\end{thm}
\begin{proof} The proof is similar to the proof of Theorem \ref{mixed} below where we provide more details. 
Denote by $D$ the following diagram: the vertices are $(X, i)$ where $X$ is an object of $\cC$ and $i\in \Z$ and arrows are $f: (X, i)\to (X', i)$ for each morphism $ X\to X'$ and $i\in \Z$. We then obtain the universal representation $\Delta : D\to \Ab_R(D)$, see Theorem \ref{unrep} and Definition \ref{unrepdef}. Consider the thick subcategory $\ll \im (\Delta_{gf} - \Delta_{g} \Delta_{f}),  \im (\Delta_{id} -  id)\gg\ \subset \Ab_R(D)$ generated by the functoriality conditions on the diagram.  Let $\cA(\cC)$ be the resulting quotient of $\Ab_R(D)$. 
We get $H\in \Hom (\cC, \cA(\cC))$ given by the image of $\Delta$ under the projection. For any $R$-linear abelian category $\cA$, we regard $K\in \Hom (\cC, \cA)$ as a representation of $D$ in $\cA$ and by the universality Theorem \ref{unrep} we have that $\Hom (\cC, \cA)\subset \Rep (D, \cA)\cong \Ex_R (\Ab_R(D), \cA)$. Therefore, an homology $K$ yields an exact functor $F_K: \Ab_R(D)\to\cA$ which factors through $\cA(\cC)$ inducing an equivalence of categories
$$K\leadsto r_K : \Hom (\cC, \cA)\cong \Ex_R (\cA(\cC), \cA)$$
which is natural in $\cA$. This proves the universality claim.

For any functor $F:\cC\to \cD$ we obtain a functor $\Hom (\cD, \cA)\to \Hom (\cC, \cA)$ by composition $K\leadsto KF$. For $\cA = \cA(\cD)$ the universal homology $K=H^\cD\in \Hom (\cD, \cA(\cD))$ yields a homology $H^\cD F\in \Hom (\cC, \cA)$ and therefore, by universality, we obtain an exact functor $r_F\df r_{FH^\cD}: \cA(\cC)\to \cA(\cD)$ as claimed. 
\end{proof}
\begin{lemma}\label{genpure}
The category $\cA(\cC)$ is generated by the set of objects $H_i(X)$ for $X\in \cC$ and $i\in \Z$. In particular, every object of  $\cA(\cC)$ is a sub-quotient of a finite product of  $H_i(X)$  for $X\in \cC$ and $i\in \Z$.
\end{lemma}
\begin{proof} Let $\cS\subseteq  \cA(\cC)$ be any abelian full subcategory containing the set $S\df \{H_i(X)\mid X\in \cC, i\in \Z \}$.  The universal homology of Theorem \ref{pure} restricts to $H: \cC\to \cS$, an homology in the abelian category $\cS$ such that under the exact inclusion $\iota_S : \cS\into \cA(\cC)$ gives back the universal homology. Therefore, by the universality Theorem \ref{pure}, we obtain an exact functor $r_H: \cA(\cC)\to \cS$ such that $r_H(H) = H$, hence $r_H\iota_\cS = id_S$ and  $\iota_\cS r_H = id_{\cA (\cC)}$ which corresponds to the universal homology. Therefore  $\cS=\cA(\cC)$ and the first claim follows. Since the full subcategory of $\cA(\cC)$ given by sub-quotients of finite direct sums of elements of $S$  is an abelian sub-category the second claim follows from the first. 
 \end{proof}
\begin{remark}\label{circ}
Note that there is an equivalence $\cA(\cC)\cong \Ab_R(D)^\circ$ by considering the decorated diagram $(D, \circ)$ of Example \ref{decex} a) in the proof of Theorem \ref{pure} including the functoriality conditions. Also note that, for each $k\in \Z$, we have the $k$-component functor $H_k: \cC\to \cA$ yielding a unique exact functor $F_k: \Ab_R (R\cC^+)\to \cA$ where $R\cC^+$ is the additive envelope in the proof of Proposition \ref{repadd}; we thus obtain an equivalence
$$\cA (\cC)\longby{\simeq} \prod \Ab_R (R\cC^+)\df \gr \Ab_R (R\cC^+)$$
where the product is taken over $\Z$-indexed copies of Freyd's category $\Ab_R (R\cC^+)$. 
\end{remark}
Dually, for cohomology theories. Actually, passing to duals in the proof of Theorem \ref{pure} we just have an application of Theorem \ref{dualrep}:
\begin{cor}\label{dualhom}
For any category $\cC$ the functor $\Coh (\cC, - )$ is representable and the universal cohomology $H^{op}\df\{ H^i=H_i\}_{i\in \Z} :\cC^{op}\to \cA(\cC)^{op}$ is the opposite of the universal homology.\end{cor}
Starting from this abelian category $\cA(\cC)$ we can further take quotients imposing more axioms and representing decorated homology theories on $\cC$. For example, considering point objects of $\cC$, \eg objects of dimensions zero or (weakly) final objects, we have:
\begin{axiom}[Point axiom for homology] \label{pointax}
For $\cC$ with a set of point objects $\{*_k\}_{k\in I}$  we say that an homology $H\in \Hom (\cC, \cA)$ satisfies the \emph{point axiom} if $H_i (*_k)= 0$ for $i \neq 0$ and $k\in I$.
\end{axiom}
Restricting to homologies satisfying the point axiom we get a functor  $\Hom_{\rm point} (\cC, - )$ and we have: 
\begin{propose}\label{point}
If $\cC$ is a category with point objects then $\Hom_{\rm point} (\cC, - )$ is representable, \ie there is a universal homology satisfying the point axiom $$H = \{ H_i\}_{i\in \Z} :\cC\to \cA^{\rm point}(\cC)$$
with values in $\cA^{\rm point}(\cC)$ a quotient of $\cA(\cC)$. If $F:\cC\to \cD$ is a functor sending point objects of $\cC$ to point objects of $\cD$, $H^\cC$ and $H^\cD$ are the universal homologies in $\cA^{\rm point}(\cC)$ and $\cA^{\rm point}(\cD)$, respectively, then there is a unique $R$-linear exact functor $r_F : \cA^{\rm point}(\cC)\to \cA^{\rm point}(\cD)$ such that $r_FH^\cC =H^\cD F$.
\end{propose}
\begin{proof}
The same arguments in the proof of the Theorem \ref{pure} applies here
defining $\cA^{\rm point}(\cC)$ to be the quotient of $\cA(\cC)$ by the thick subcategory $\ll H_i(*)\gg$ for $i\neq 0$ and $*$ any point object. The homology in $\cA^{\rm point}(\cC)$ under the projection is the universal homology satisfying the point axiom by the universality of the quotient. Moreover, since $\Hom_{\rm point} (\cD, - )\to\Hom_{\rm point} (\cC, - )$ the induced functor is granted.\end{proof}

\begin{example}[Universal (co)homology on the point category] \label{pointex}
Let $\cC = {\bf 1}$ be the one point category, \ie the category with one point object and one morphism, necessarily the identity of this object.  As a diagram ${\bf 1} = \{\circlearrowleft \}$ is the one edge diagram, having one vertex $\{*\}$ and one non-empty edge $e: *\to *$. Now $H\in \Hom ({\bf 1}, \cA)$ is a $\Z$-indexed family $\{A_i \}_{i\in \Z}$ of objects of $\cA$ and  $H\in \Hom_{\rm point}({\bf 1}, \cA)$ is given by one object of $\cA$ (which also corresponds to the constant homology on that object). Therefore, since $\Ab_R$ of Definition \ref{unrepdef} is the $R$-linear abelian category representing the  functor $\cA\leadsto \cA = \text{$R$-Mod }(\cA): \Ex_R\to \Cat$, \ie given by an equivalence
$$\Ex_R(\Ab_R, \cA)\cong \cA$$
which is natural in $\cA$, then $ \Ab_R\cong \cA^{\rm point}({\bf 1})$ induced by $\cA \cong  \Hom_{\rm point} ({\bf 1}, \cA)$. 
Comparing with the proof of Theorem \ref{pure}, consider the decorated diagram $(\{\circlearrowleft \}, \circ)$ which is the one edge diagram ``with identity'' \ie with a single commutativity relation imposing that the empty edge shall be identified with $e$; actually,  the morphisms in the path category $\overline{\{\circlearrowleft \}}$ are $ (*), (e),(e^2), \ldots$ but we get $\{\circlearrowleft \}^\circ = {\bf 1}$ and a $\circ$-representation of $\{\circlearrowleft \}$ in $\cA$ is a functor from ${\bf 1}$ to $\cA$. We then obtain that $\Rep^\circ (\{\circlearrowleft \}, \cA)\cong \cA$ (note that $(\{\circlearrowleft \}, \circ)$ is a categorical diagram) and we have the following chain of equivalences, \cf Remark \ref{circ}, Example \ref{pointR} and Definitions \ref{hieraticdef} \& \ref{unrepdef}
 $$\Ab_R\cong \iRmod\cong \cA^{\rm point}({\bf 1})\cong \Ab_R (\{\circlearrowleft \})^\circ $$ 
where the universal hieratic $R$-module $|R|\in \iRmod$ given by the image of $R$, corresponds to $H_0(1) \in \cA^{\rm point}({\bf 1})$,  the universal homology of the point. We also have natural equivalences 
 $$\cA = R\mbox{-}{\rm Mod}(\cA)   \cong \Lex_R (\Rmod^{op}, \cA)\cong \Ex_R (\Ab_R , \cA)$$
 for any $\cA$, see Theorem \ref{Freyd}  (\cf \cite[\S 1.3]{BVP}). If $\Rmod$ is abelian, \ie for $R$ coherent, the ``evaluation at the ring $R$'' $R$-linear exact functor 
 $$r_R: \iRmod\onto \Rmod \subseteq \RMod$$
 yields $\Rmod$ as a quotient of $\iRmod$. In fact, $ R^+\subseteq \iRmod/\ker r_R =  \Rmod$ (\eg see also \cite[Lemma 6.3]{PR}).  
 If every extension splits in $\Rmod$ and we have a duality $( -) ^\vee: \Rmod^{op}\by{\simeq} \Rmod$,  \eg if $R$ is a field,  then we have the following factorisation 
 $$\xymatrix{\Rmod^{op}\ar[r]^{ } \ar[d]_{( - )^\vee}^{||} & \iRmod  \ar[dl]^{r_R} \\ \Rmod}$$
 induced by the universal property and 
$$\Ex_R (\Rmod, \cA) = \Add_R (\Rmod, \cA) \cong\cA$$ 
for $\cA$ any $R$-linear abelian category. Therefore $r_R : \iRmod\by{\simeq} \Rmod$ for such a ring and the quasi-inverse to $r_R$ is induced by the Yoneda  embedding $\Rmod^{op}\by{\simeq} \iRmod$. In general,  from the duality Corollary \ref{dualinv}, we always have that $\Ab_R^{op}= \Ab_R$ whence $\iRmod^{op}\cong \iRmod$.  If $r_R$ is an equivalence then $\Rmod^{op}\cong \Rmod$ and the coherent ring $R$ is absolutely pure, \ie $\Ext (M,R)=0$ for $M\in \Rmod$, see \cite{GA}.
For a coherent domain $R$ with field of fractions $K$ we then have that the quotient $r_R$ induces
$$\xymatrix{\iRmod\otimes_RK\ar[r]^{\hspace*{0.5cm}=} \ar@/^2.3pc/[rr]^{r_R\otimes K } & \iKmod \ar[r]^{\simeq\hspace*{1.7cm}} & \Kmod = \Rmod\otimes_R K}$$
and $r_R\otimes K : \iRmod\otimes_R K \longby{\simeq}\Kmod$ is an equivalence. In particular, for $R=\Z$ and $K=\Q$ we obtain that $\iZmod\otimes_\Z\Q \cong \Qmod$. The simple objects of  $\iZmod$ are known. There is one for every $\Z$-module  $\Z/p^n$ where $p^n$ is a prime power, $n\geq 1$: it can be described as the representable functor $ (\Z/p^n,-)$ modulo its radical, see \cite[Thm.\,8.55]{JL} and \cite[Prop.\,14.2.6 \& Thm.\,14.2.7]{KH}. Moreover, the representable functors form a system of generators consisting of finitely presented projectives, \cf \cite[Thm.\,B.8]{JL}.\footnote{Thanks to M. Prest for pointing out these facts about $\Ab_\Z$ and the reference \cite{JL}.} 
 \end{example}
 \begin{remark}
  Note that constant homologies are equivalent to homologies on the point category satisfying the point axiom and this implies that the constant homologies subfunctor is representable by  $\Ab_R\cong \iRmod$.
For a category $\cC$ we also have a unique functor $\cC\to {\bf 1}$, as ${\bf 1}$ is final in ${\rm Cat}$, which yields $\pi : \cA^{\rm point}(\cC)\to \iRmod$ and $ \cA\cong \Hom_{\rm point}({\bf 1}, \cA) \cong \Hom_0 (\cC, \cA)$ given by those homologies $X\leadsto H_0 (X)=A$ for a fixed object $A\in \cA$ and $H_i (X)=0$ for $i\neq 0$.  Moreover, we have a section functor $r_{H_0(*)} :  \iRmod\to \cA^{\rm point}(\cC)$ given by $|R|\leadsto H_0(*)$, also induced by ${\bf 1}\to \cC$ sending the unique point object to a point object. 
 \end{remark}
 \begin{propose} \label{kproj}
Let $\cC$ be a category with an initial or final object. Then for each integer $k\in\Z$ there is a quotient
$$\pi_k:\cA(\cC)\onto \iRmod$$
with a section $\iota_k :  \iRmod\to \cA(\cC)$ where $\iRmod$ is the abelian category of Definitions \ref{hieraticdef} \& \ref{unrepdef}, \cf Example \ref{pointex}. For $k=0$ the functor $\pi_0$ factors through $\cA^{\rm point}(\cC)$ with a section.
\end{propose}
 \begin{proof} For each $X\in \cC$ let $\int : X\to 1$. We then get $\int_i : H_i(X)\to H_i(1)$ where $H = \{ H_i\}_{i\in \Z} :\cC\to \cA(\cC)$ is the universal homology. 
For each fixed integer $k$ consider the thick subcategory of $\cA(\cC)$ generated by $H_i(X)$ for $i\neq k$, $\ker \int_k$ and $\coker \int_k$ for $\int_k : H_k(X)\to H_k(1)$ and all $X\in \cC$.  Denote $\cA(\cC)_k$ the resulting quotient of  $\cA(\cC)$. The category  $\cA(\cC)_k$ represents the subfunctor $\Hom_k (\cC, - )\subset \Hom (\cC, - )$ determined by  the trivial homology and those non-trivial homologies $H = \{H_i\}$ such that $H_i(X)\neq 0$ implies $i=k$ and $\int_k :H_k(X)\cong H_k(1)$ for all $X\in \cC$. Clearly $\cA \cong \Hom_k (\cC, \cA )$ by sending $A \leadsto H_k(1)\df A$ and $\Ab_R\cong \iRmod\cong \cA(\cC)_k$. We then get a quotient $\pi_k:\cA(\cC)\onto \iRmod$ for each integer $k\in \Z$. Now $H_k(1)\in\cA(\cC)$ yields an exact functor $\iota_k : \iRmod\cong\cA(\cC)_k\to \cA(\cC)$ whose composition with the projection $\pi_k : \cA(\cC)\to \cA(\cC)_k\cong \iRmod$ is the identity. For $k=0$ we can argue in $\cA^{\rm point}(\cC)$ and get to the same conclusion.
 \end{proof}
\begin{example}[Universal group representations]\label{grprep}
Let $\cC = N$ be a monoid or group regarded as a category with one point object: note that it is not initial nor final. For any $R$-linear abelian category $\cA$ we have that 
$$\Hom_{\rm point}(N, \cA)=\Hom_{\rm point}(R[N], \cA)\df {\rm Rep}_R(R[N], \cA)$$ is the usual $R$-linear abelian category of representations of the monoid $N$ or, equivalently, of the free $R$-module $R[N]$, \ie an object $\rho\in {\rm Rep}_R(R[N], \cA)$ is an $R$-homomorphism $\rho: R[N]\to \End (A)$ for $A\in \cA$. Therefore, the universal homology $N\to\cA^{\rm point}(N)$ yields the universal representation $h: R[N]\to \End (H)$ where $H$ is an object and the exact functor $\pi: \cA^{\rm point}(N)\to\iRmod$ such that $\pi (H)=|R|$: for any $\rho$ exists an exact functor $r_\rho: \cA^{\rm point}(N)\to \cA$ such that $\rho=r_\rho h$ factors through $h$ via $r_\rho$. \end{example}

\subsection{Relative homology theory} Assume given a category $\cC$ along with a distinguished subcategory. Denote $\cC^{\square}$ the category whose objects are the distinguished morphisms and whose arrows are commutative squares in $\cC$. Denote $(X, Y)$ a distinguished morphism $f: Y\to X$, \ie an object of $\cC^{\square}$; a morphism is a pair $\gamma = (h, k): (X, Y)\to (X', Y')$ where $h: X\to X'$ and $k: Y\to Y'$ are morphisms of $\cC$ such that $hf=f'k$ if $f$ and $f'$ are distinguished morphisms defining $(X,Y)$ and $(X', Y')$ respectively. For a triple $(X,Y)$, $(X,Z)$ and $(Y,Z)$ in $\cC^{\square}$ we mean the following factorisation diagram of $\cC^{\square}$ 
$$
\xymatrix{Z\ar[r]^-{f}  \ar[d]_-{||} & Y \ar[d]^-{g} \\
Z \ar@{>}[r]^-{gf}\ar[d]_-{f}  & X\ar[d]_-{||}\\
Y \ar@{>}[r]^-{g} & X}
$$
where $f : Z \to Y$ and $g: Y \to X$ are distinguished morphisms and so it is $gf: Z\to X$ defining $(X, Z)$. We have canonical morphisms $\alpha \df (g, id): (Y,Z)\to (X,Z)$ and $\beta \df (id, f): (X,Z)\to (X,Y)$ whose composition is the morphism $(g, f) = \beta\alpha: (Y,Z)\to (X,Y)$ given by the outer square in the previous diagram. 

Consider an $R$-linear abelian category $\cA$ and a collection of objects as $i\in \Z$ varies
$$(X,Y)\in \cC^\square\leadsto H_i (X,Y)\in \cA$$
and for each morphism $\gamma: (X, Y) \to (X', Y')$ in $\cC^\square$  a morphism $\gamma_i : H_i(X, Y) \to H_i(X', Y')$  in $\cA$ together with an extra connecting morphism $\partial_i :H_i(X,Y) \to H_{i-1}(Y,Z)$ associated with the triple $(X,Y)$, $(X,Z)$ and $(Y,Z)$ in $\cC^{\square}$, in such a way that 
\begin{itemize}
\item[{\it i)}] $H = \{H_i\}_{i\in\Z}$ is a family of functors  $H_i : \cC^{\square}\to \cA$,
\item[{\it ii)}] we have the long exact sequence of the triple
$$\cdots\to H_i(Y,Z) \longby{\alpha_i}  H_i(X,Z) \longby{\beta_i}  H_i(X,Y) \longby{\partial_i} H_{i-1}(Y,Z)\to H_{i-1}(X,Z) \to\cdots $$ and
\item[{\it iii)}] this long exact sequence is natural with respect to $\partial$-cubes. The latter property being explained as follows: considering a $\partial$-cube, \ie the following 
$$\xymatrix{(Y,Z)\ar[r]^-{\beta\alpha}  \ar[d]_-{\gamma} & (X,Y) \ar[d]^-{\delta} \\
(Y',Z')\ar@{>}[r]^-{\beta'\alpha'} & (X',Y')}$$
commutative square in $\cC^\square$, we get an induced morphism $\kappa :(X,Z)\to (X',Z')$ such that $\kappa\alpha = \alpha'\gamma$ and $\delta\beta = \beta'\kappa$ and we obtain the following induced diagram in $\cA$ 
$$\xymatrix{\cdots \ar[r]^-{} & H_i(Y,Z)\ar[r]^-{\alpha_i}  \ar[d]^-{\gamma_i} & H_i(X,Z) \ar[d]^-{\kappa_i}\ar[r]^-{\beta_i} & H_i(X,Y) \ar[r]^-{\partial_i}  \ar[d]^-{\delta_i} &  H_{i-1}(Y,Z)\ar[r]^-{}  \ar[d]^-{\gamma_{i-1}} & \cdots\\
\cdots \ar[r]^-{} & H_i(Y',Z')\ar[r]^-{\alpha_i'}  & H_i(X',Z') \ar[r]^-{\beta_i'} & H_i(X',Y') \ar[r]^-{\partial_i'}  &  H_{i-1}(Y',Z') \ar[r]^-{}  & \cdots}$$
for which, the naturality requirement is that $\gamma_{i-1}\partial_i = \partial_i' \delta_i$ (note that we thus obtain a morphism of long exact sequences as the other two squares are commutative by functoriality). 
\end{itemize} 

A morphism $\varphi : H\to H'$ for $H$ and $H'$ having values in $\cA$, shall be a collection of natural transformations $\varphi_i : H_i\to H'_i$ which are compatible with the $\partial_i$ \ie such that the following
$$\xymatrix{H_i(X,Y)\ar[r]^-{\partial_i}  \ar[d]_-{\varphi_i} & H_{i-1}(Y,Z) \ar[d]^-{\varphi_{i-1}} \\
H_i'(X,Y)\ar@{>}[r]^-{\partial_i} & H_{i-1}'(Y,Z) }$$
is commutative for any triple.
All this is modelled on the regular homology theory \cite{BV}, see \cite[\S 3.1]{BV} for details. 
\begin{defn} \label{relative}
A \emph{relative homology} $H$ on $\cC^{\square}$ with values in $\cA$ shall be given by the set of data and conditions listed in i) -- iii) above. We shall denote by $\Hom (\cC^\square, \cA)$ the category of homologies on $\cC^\square$ with values in $\cA$.  We shall adopt the same terminology as for complexes by saying that $H$ is bounded above, below and concentrated in non positive or non negative degrees, \eg this latter means that $H = \{H_i\}_{i\in\Z}$ is such that $H_i=0$ for $i<0$. 

Dually, there is a corresponding notion of \emph{relative cohomology}: it is given by a family of functors $H = \{H^i\}_{i\in\Z} : \cC^{\square , op}\to \cA$ together with the connecting morphism $\partial^i : H^{i}(Y,Z)\to H^{i+1}(X,Y)$ and the long exact sequence of the triple which satisfies the naturality condition with reversed arrows. Denote $\Coh (\cC^\square, \cA)$ the category of relative cohomologies. 
 \end{defn}
 If $H = \{H_i\}_{i\in\Z} : \cC^{\square}\to \cA$ is a relative homology then $H^{op}$ given by $\{H^i=H_i \}_{i\in\Z} : \cC^{\square , op}\to \cA^{op}$ is the opposite relative cohomology and $\Hom (\cC^\square, \cA^{op})\cong \Coh (\cC^\square, \cA)$.
Relative homology and cohomology is fitting Eilenberg-Steenrod axiomatic approach \cite{ES}.
\begin{remarks} \label{partialrmk}
a) Relative homology is also a generalisation of Grothendieck exact covariant $\partial^*$-functor \cite[\S 2.1]{Gr} and Cartan-Eilenberg connected sequence of functors \cite[Chap.\,III]{CE}, see Lemma \ref{partial} below for a comparison. 

b) Note that we can make up the category $\cC^\blacksquare$ of $\partial$-cubes. This is the full subcategory of the category of functors from ${\bf 2}$ to $\cC^\square$ whose objects are the triples, \ie the morphisms $\partial \df (g, f) = \beta\alpha: (Y,Z)\to (X,Y)$, and the morphisms are $\partial$-cubes.   
\end{remarks}
For an $R$-linear exact functor $F:\cA \to \cA'$ and $H\in \Hom (\cC^\square, \cA)$ we obtain $FH = \{FH_i\}_{i\in \Z} \in \Hom (\cC^\square, \cA')$ by composition. This is making up a 2-functor 
$$\cA \leadsto  \Hom (\cC^\square, \cA): \Ex_R \to \Cat$$
from the category $\Ex_R$ of $R$-linear abelian categories and exact functors to categories. 
\begin{defn}[Nori diagram]
Call \emph{Nori diagram} $D^\square$ of $\cC^\square$ the following diagram: 
\begin{itemize}
\item[{\it i)}] the vertices are $(X,Y, i)$ where $(X, Y)$ is an object of $\cC^\square$ and $i\in \Z$;
\item[{\it ii)}] the arrows are $\gamma: (X, Y, i)\to (X', Y', i)$ for each morphism 
$(X, Y)\to (X',Y')$ in $\cC^\square$ and $\partial: (X,Y, i) \to (Y,Z, i-1)$ the extra connecting edge of the triple and $i\in \Z$, \ie corresponding to the morphism $(g, f) : (Y,Z)\to (X,Y)$ for  $f : Z \to Y$ and $g: Y \to X$ distinguished morphisms.
\end{itemize}
\end{defn}
We have:
\begin{thm}[Universal relative homology] \label{mixed}
For any category $\cC$ along with a distinguished subcategory the functor $\Hom (\cC^\square, - )$ is representable, \ie  there is a universal relative homology $$H = \{ H_i\}_{i\in \Z} :\cC^{\square}\to \cA_\partial(\cC)$$
with values in $\cA_\partial(\cC)$ a $R$-linear abelian category. If $F:\cC\to \cD$ is a functor compatible with distinguished subcategories, \ie inducing a functor $F^\square: \cC^{\square}\to \cD^{\square}$, then there is a unique $R$-linear exact functor $F_\partial : \cA_\partial (\cC)\to \cA_\partial (\cD)$ such that $F_\partial H^\cC =H^\cD F^\square$ where $H^\cC$ and $H^\cD$ are the universal relative homologies.

\end{thm}
\begin{proof} Applying Thoerem \ref{unrep} to the Nori diagram $D^\square$ of  $\cC^\square$ we  obtain the universal representation $\Delta : D^\square\to \Ab_R(D^\square)$. Consider the thick subcategory of $\Ab_R(D^\square)$ generated by the following objects: 
\begin{itemize}
\item $\im (\Delta_{id} -id)$, $\im (\Delta_{\gamma\delta} - \Delta_{\gamma}\Delta_{\delta})$, for the functoriality axiom, 
\item $\im (\Delta_{\gamma}\Delta_{\partial} - \Delta_{\partial'}\Delta_{\delta})$, associated to the $\partial$-cubes for the naturality axiom and 
\item $\im \Delta_{\beta}\Delta_{\alpha}$, $\im \Delta_{\partial}\Delta_{\beta}$, 
$\im \Delta_{\alpha}\Delta_{\partial}$ for the complex associated to the triple. 
\end{itemize}
Let $\Ab_R(D^\square)^c$ be the resulting quotient abelian category and let $H^c$ be the image of $\Delta$ under the projection $\Ab_R(D^\square)\onto \Ab_R(D^\square)^c$. In this category $\Ab_R(D^\square)^c$ we now have the following complex associated to each triple 
$$H^c_* : \cdots\to H_i^c(Y,Z) \longby{\alpha_i}  H_i^c(X,Z) \longby{\beta_i}  H_i^c(X,Y) \longby{\partial_i} H_{i-1}^c(Y,Z)\to \cdots $$ 
We can impose the exactness of these complexes for all triples by taking homologies $H_i(H^c_*)$ of all these complexes
 and pass to a further quotient $\cA_\partial(\cC)\df \Ab_R(D^\square)^c/ \ll H_i(H^c_*)\gg$. We thus have  
$$\xymatrix{D^\square\ar[r]^{\Delta\ \ \ }& \Ab_R(D^\square)\ar[r]^{} \ar@/^1.7pc/[rr]^{\pi } & \Ab_R(D^\square)^c\ar[r]^{} &\cA_\partial(\cC)}$$
and we obtain  $(X, Y, i)\leadsto H_i(X, Y)\df \pi (\Delta_{(X, Y, i)})$ a representation $H \in \Rep (D^\square, \cA_\partial(\cC))$ given by the image of $\Delta$ under the projection: by construction $(X, Y)\in \cC^\square \leadsto H_i(X, Y)\in \cA_\partial(\cC)$ is functorial and satisfies the other axioms providing   $H = \{ H_i\}_{i\in \Z} :\cC^{\square}\to \cA_\partial(\cC)$ a relative homology. 

Now, for any $R$-linear abelian category $\cA$, observe that we can regard any relative homology in $\Hom (\cC^{\square}, \cA)$ as an object of  $\Rep (D^{\square}, \cA)$; moreover, a morphism of  homologies is precisely a morphism as representations of $D^\square$ in $\cA$  showing that $\Hom (\cC^{\square}, \cA)\subset \Rep (D^{\square}, \cA)$ is a full subcategory. By Theorem \ref{unrep}, we have that $\Rep (D^\square, \cA)\cong \Ex_R (\Ab_R(D^\square), \cA)$ so that a relative homology $K$ yields a unique (up to natural equivalence)  $R$-linear exact functor $F_K : \Ab_R(D^\square)\to \cA$ such that $F_K (\Delta)= K$ as representations, \ie the following
$$\xymatrix{D^{\square}\ar[r]^{\Delta \  \  \ } \ar[d]_K & \Ab_R(D^\square)  \ar[dl]^{F_K} \\  \cA}$$
is commutative. 
This functor $F_K$ factors through the quotient $\cA_\partial(\cC)$ providing an exact (not necessarily faithful) functor  
$r_K: \cA_\partial(\cC)\to \cA$. In fact, since $F_K$ is exact it is sending to zero all objects in the kernel of $\Ab_R(D^\square)\onto \Ab_R(D^\square)^c$, \eg $F_K(\im \Delta_{\beta}\Delta_{\alpha}) = \im \beta_i \alpha_i = 0$ since $ \beta_i \alpha_i : K_i(Y,Z)\to K_i (X, Y)$ is the zero morphism from the chain complex condition; therefore, $F_K$ yields an exact functor $F_K^c : \Ab_R(D^\square)^c\to \cA$ and since $F_K^c (H_i(H^c_*)) = H_i(F_K (H^c_*)) = H_i(K_*)=0$ we get an induced functor $r_K: \cA_\partial(\cC)\to \cA$ as claimed. Actually, 
$\ker \pi \subseteq  \ker F_K$ and if we consider the composition of the quotient $$\cA_\partial(\cC)=  \Ab_R(D^\square)/\ker \pi \onto \Ab_R(D^\square)/\ker F_K$$ with the faithful exact functor $ \Ab_R(D^\square)/\ker F_K\into \cA$ we get $r_K$ as claimed. 
Note that any exact functor $F: \cA_\partial(\cC)\to \cA$ yields $FH\in \Hom (\cC^{\square}, \cA)$ and $r_{FH}\cong F$ by unicity of $F_{FH}$.
We thus obtain the refined equivalence
$$K\leadsto r_K: \Hom (\cC^{\square}, \cA)\cong \Ex_R (\cA_\partial(\cC), \cA)$$
which is natural in $\cA$. Finally, $H\in \Hom (\cC^\square, \cA_\partial(\cC))$ yields $r_H = id$ by construction and therefore $H$ is the universal relative homology. 

For a functor $F^\square: \cC^{\square}\to \cD^{\square}$ we have that $H^\cD F^\square\in \Hom (\cC^\square, \cA_\partial(\cD))$ is a relative homology and therefore $F_\partial\df r_{H^\cD F}$ yields the exact functor as claimed. 
\end{proof}
\begin{remark}\label{boundrel}
The same arguments in the proof of the Theorem \ref{mixed} provide representability of the sub-functors of  $\Hom (\cC^\square, - )$ given by relative homologies whose degrees are concentrated in an interval $[a, b]$ where $a$ and $b$ are integers such that $a+1<b$ (including $a =-\infty$ or $b=+\infty$). We can add $H_i^c(X, Y)$ for $i\notin [a,b]$ in the list of generators of the thick subcategory of  $\Ab_R (D^\square)^c$ when imposing the exactness of the complexes $H^c_*$ and get that $H_i (X, Y)= 0$ for $i\notin [a,b]$;   we obtain decorated abelian categories $\cA_\partial^{[a, b]}(\cC)$  along with a quotient  $\cA_\partial (\cC)\onto \cA_\partial^{[a, b]}(\cC)$. For example, if $a\geq 0$ we have that $$H_a (Y, Z)\to H_a (X, Z)\to H_a (X, Y)\to 0$$ is exact in  $\cA_\partial^{[a, b]}$ for the universal relative homology sequence of a triple. 
\end{remark}
Note that we can deal with homologies in Grothendieck categories: we get a quotient $\Ind \Ab_R(D^\square)\onto\Ind \cA_\partial(\cC)$ and any exact functor $\cA_\partial(\cC)\to \cA$ in a Grothendieck category uniquely extend to $\Ind \cA_\partial(\cC)$, see Proposition \ref{indadjoint}.
\begin{cor} \label{induni}
The relative homology $H = \{ H_i\}_{i\in \Z} :\cC^{\square}\to \Ind \cA_\partial(\cC)$ given by the image of the universal relative homology $H$ under 
$\cA_\partial(\cC)\into \Ind \cA_\partial(\cC)$ is the universal Grothendieck relative homology, \ie it represents the 2-functor $\Hom (\cC^\square, -)$ over Grothendieck categories. 
\end{cor}
Passing to cohomologies we see that the opposite diagram $D^{\square , op}$ in the proof of Theorem \ref{mixed} is providing the universal relative cohomology and applying Theorem \ref{dualrep} we obtain:
\begin{cor}\label{dualrelhom}
The functor of relative cohomologies $\Coh (\cC^{\square}, - )$ is representable and the universal relative cohomology $H^{op} = \{ H^i\}_{i\in \Z} :\cC^{\square , op}\to \cA_\partial (\cC)^{op}$ is the opposite of the universal relative homology.\end{cor}
Recall \cite{BV} where the (mixed) regular homology theory $\T$ has been introduced and the resulting abelian category of constructible $\T$-motives $\cA[\T]$  is  presented as the (Barr) exact completion of the syntactic category $\cC_\T^{\rm reg}$ of the regular theory. Recall that $\T$-motives are given by $\Ind \cA[\T]$. We already noted in \cite[Thm.\,2.7 \& Cor.\,2.9]{BVP} that $\cA[\T]$ can be recovered as a quotient of Freyd's abelian category $\Ab (D^\square)$ for $D^\square$ the Nori diagram of  $\cC^\square$. We here show a very simple and direct argument to see this.
\begin{propose}\label{Tmotives}
There is a canonical equivalence $$\cA[\T]\cong \cA_\partial(\cC)$$ of abelian categories under which the universal model $H^\T$ corresponds to $H$  the universal  relative homology. Moreover, $\T$-models and relative homologies coincide. For any relative homology $K$, if $\T_K$ is the theory obtained by adding all regular axioms which are valid in the model $K$ then there is an equivalence  
$$\cA[\T_K]\cong \cA (K)\df \Ab_R(D^\square)/\ker F_K$$
under which the universal model corresponds to the universal relative homology generated by $K$.
\end{propose}
\begin{proof} Recall \cite[Prop.\,4.1.3]{BV} that, for any abelian category $\cA$, we have a natural equivalence $\TMod (\cA)\cong \Ex (\cA [\T], \cA)$ where $\TMod (\cA)$ is the category of models in $\cA$ of the theory $\T$.
By construction, see \cite[\S 3]{BV}, we have that $$\TMod (\cA)= \Hom (\cC^{\square}, \cA)$$ these categories coincide. Whence we obtain that $\Ex (\cA [\T], \cA)\cong \Ex (\cA_\partial(\cC), \cA)$ by the Theorem \ref{mixed}, for any abelian category $\cA$. Therefore we obtain the claimed equivalence. The same argument applies to $\T_K$ and the universal model for both $\T$ and $\T_K$ correspond to the respective universal relative homologies by unicity of the representing objects.
\end{proof}
The following is a reformulation of \cite[Lemma 3.1.1]{BV}.
\begin{lemma}\label{purity}
If  $H\in \Hom (\cC^\square, \cA)$ then $H_i(X, Y) =0$ for all $Y \cong X$ distinguished  isomorphisms and $i\in \Z$.
\end{lemma}
\begin{proof} If  $f: Y \by{\simeq} X$ is a distinguished  isomorphism then $(X,Y) \cong (X,X)$ given by
$$\xymatrix{Y\ar[r]^-{f}  \ar[d]_-{f} & X \ar[d]^-{||} \\
X \ar@{>}[r]^-{=} & X}$$
which yields $H_*(X,Y)\cong H_*(X,X)$ by functoriality. From the exactness of 
$$H_*(X,X) \by{id} H_*(X,X) \by{id} H_*(X,X)$$
for $X=Y=Z$ we obtain $H_*(X,X) =0$.
\end{proof}
From Lemma \ref{purity} we see that there are no non trivial constant relative homologies, in general. Note that the universal abelian category $\cA_\partial(\cC)$ is the minimal quotient of  $\Ab_R (D^\square)$ generated by the relative homology axioms and we shall get further quotients imposing further axioms, \eg see \cite{UCTII} for ordinary theories.
\begin{remark}[Transfers]\label{trans}
 We may consider homology theories with transfers as follows.  Let $R$ be a subring of $\Q$.  Consider morphisms of $\cC$, called ``finite coverings'', assuming that
\begin{itemize}
\item[{\it i)}] finite coverings are stable by composition and all isomorphisms are finite coverings, so that we have $\cF\subset\cC$ a subcategory of finite coverings; 
\item[{\it ii)}] for any finite covering $f:\tilde X\to X$ we have a well defined degree $\deg (f)\in R$ such that  $\deg (gf)= \deg (g) \deg (f)$ and $\deg (id_X)=1$. \end{itemize}
We may also assume that a distinguished morphism which is a finite covering is an isomorphism. 

An homology with transfers $H\in \Hom_\tr (\cC, \cA )$ is an homology for which we require the existence of a transfer morphism
$\tr_i^f: H_i(X)\to H_i(\tilde X)$ which depends functorially on the finite covering $f:\tilde X\to X$  and such that  
$f_i\tr_i^f : H_i(X)\to H_i(X)$ is the homothety determined by  $\deg (f)$ where $f_i :H_i(\tilde X)\to H_i(Y)$ is the morphism given by functoriality. We clearly can add  $\cF$ defining a category with transfers and obtaining a universal homology  with transfers $H^\tr : \cC \to \cA^\tr(\cC)$ by representing the functor $\Hom_\tr (\cC, - )$ given by homologies with transfers. This $H^\tr$ is obtained by considering the diagram with transfers $D_\tr$, adding to the diagram $D$ in the proof of Theorem \ref{pure} the edges $(X,  i)\to (\tilde X, i)$ for each $f: \tilde X\to X\in \cF$. Moreover, universal relative homology with transfers $H^\tr : \cC^\square \to \cA_\partial^\tr(\cC)$ exists by representing the functor $\Hom_\tr (\cC^\square, - )$ on relative homologies.  In fact, we may also add relative versions of transfers asking for compatibility with the long exact sequence of a triple and adding transfers to the Nori diagram $D_\tr^\square$ we can adapt the proof of Theorem \ref{mixed}. We get an $R$-linear exact functor 
$r_{H^\tr}:\cA_\partial (\cC)\to \cA_\partial^\tr(\cC)$. 
\end{remark}
\subsection{Homology versus relative homology}
For $\cC$ with an initial object $0$ assume that $0\to X$ is distinguished for all $X\in \cC$. We then obtain a functor 
$\cC \to \cC^{\square}$  given by  $X\leadsto (X, 0 )$. Assume given a relative homology  $H\in \Hom (\cC^{\square}, \cA )$ and set 
$$H_i(X)\df H_i(X,0)\in \cA$$ 
for the \emph{restricted} homology along $\cC \to \cC^{\square}$.
The long exact sequence of the triple $(Y, 0)\to (X, Y)$, here $Z=0$, is the following sequence in $\cA$ 
 $$\cdots\to H_i(Y) \longby{\alpha_i}  H_i(X) \longby{\beta_i}  H_i(X,Y) \longby{\partial_i} H_{i-1}(Y)\to H_{i-1}(X) \to\cdots $$
 which is named,  by convention, the long exact sequence of the pair. The restricted homology $X\leadsto H_i(X)$ yields an object  $H\in \Hom (\cC, \cA )$ and, by universality, see Theorem \ref{pure}, we obtain $r_H: \cA(\cC) \to  \cA$ the induced exact functor. 
 
 In particular, consider the universal relative homology in $\Hom (\cC^\square, \cA_\partial(\cC))$ of Theorem \ref{mixed}; by restriction along $\cC \to \cC^{\square}$ we obtain an object of $\Hom (\cC, \cA_\partial(\cC))$. We  thus obtain 
$$r_\partial : \cA(\cC) \to  \cA_\partial(\cC)$$
 an exact functor. For $H$ the universal homology $r_\partial (H)$ coincide with the restricted universal relative homology: here $H_i(X)$ is an object of $\cA(\cC)$  and $r_\partial (H_i(X))=H_i(X, 0)$, which we also denote $H_i(X)$ but regarded as an object of $ \cA_\partial(\cC)$, by a modest abuse of notation. However, note that the functor $r_\partial$ is not faithful, in general; for example, $r_\partial (H_i (0))= H_i (0,0)=0$ is vanishing, \cf Example \ref{pointexrel}. The essential image of $r_\partial$ is, by Lemma \ref{genpure}, the abelian subcategory of $\cA_\partial(\cC)$ whose objects are sub-quotients of finite products of  $H_i(X)$  for $X\in \cC$ and $i\in \Z$.
 \begin{thm}[Generating subcategory]\label{genmixed}
 The category $\cA_\partial (\cC)$ coincide with the smallest abelian subcategory which contains kernels and cokernels of morphisms between objects $H_i(X)$ for $X\in \cC$ and $i\in \Z$ and it is closed by extensions of these.  
In particular, every object of $\cA_\partial (\cC)$ is an extension of sub-quotients of $H_i(X)$ for $X\in \cC$ and $i\in \Z$. 
\end{thm}
\begin{proof} Adapting the arguments in the proof of Lemma \ref{genpure} we have that any abelian subcategory $\cS\subseteq\cA_\partial  (\cC)$, containing the set $S\df \{ H_i(X, Y)\mid (X,Y)\in \cC^\square, i\in \Z \}$ and such that the inclusion $\iota_S : \cS\into \cA(\cC)$ is  exact coincide with $\cA_\partial  (\cC)$. This follows from the universality Theorem \ref{mixed}. By the exact sequence of the pair we see that $H_i(X, Y)$ is an extension of kernels and cokernels of morphisms between objects of the form $H_i(X)$ for $X\in \cC$ and $i\in \Z$.  Therefore, if an abelian subcategory  $\cS$ contains all these then it contains $S$ and the thick subcategory generated by 
$\{ H_i(X, 0)\mid (X,0)\in \cC^\square, i\in \Z \}$ coincide with all $\cA_\partial  (\cC)$. 
 \end{proof}
Assume that a category $\cC$ is provided with an initial object $0$ and a set of point objects. 
\begin{axiom}[Point axiom for relative homology] \label{pointrelax}
We say that a relative homology $H\in \Hom (\cC^{\square}, \cA)$ satisfies the \emph{point axiom} if $H_i (*, 0 )= 0$ for $i \neq 0$  where the canonical morphism $(*,0)\df 0\to *$ from the initial object to any point object of $\cC$ is assumed to be a distinguished morphism.
\end{axiom}
If $0\to 1$ is an isomorphism and $*=1$ then the point axiom is trivially satisfied by Lemma \ref{purity}. Let $\Hom_{\rm point} (\cC^{\square}, - )$ be the functor of relative homologies satisfying the point axiom. 
\begin{propose}\label{relpoint}
The functor $\Hom_{\rm point} (\cC^{\square}, - )$ is representable, \ie there is a universal relative homology satisfying the point axiom $H = \{ H_i\}_{i\in \Z} :\cC^{\square}\to \cA_\partial^{\rm point}(\cC)$
with values in $ \cA_\partial^{\rm point}(\cC)$ a quotient of $ \cA_\partial(\cC)$.
\end{propose}
\begin{proof}
The category $ \cA_\partial^{\rm point}(\cC)$  is the quotient of  $ \cA_\partial(\cC)$ by the thick subcategory $\ll H_i(*, 0)\gg$ for $i\neq 0$ and any point object $*$ where $H$ is the universal relative homology. 
\end{proof}
We clearly get the induced exact functor 
$$r_\partial^{\rm point} : \cA^{\rm point}(\cC)\to \cA_\partial^{\rm point}(\cC)$$
\begin{example}[Universal relative (co)homology on the two objects category] \label{pointexrel} 
Let $\cC= {\bf 2}$ be the category with two objects determined by the ordinal $2\in \N$ such that $1$ is a point object. The  diagram underlying to this category is given by two vertices $0\df \emptyset$ and  $1\df \{\emptyset\}$, one arrow $0\to 1$ and the two identities $\circlearrowleft$ attached to  $0$ and $1$. If we choose ${\bf 2}$ as a distinguished subcategory we have that ${\bf 2}^\square$ has three objects $(1, 0)$,  $(0, 0)$ and $ (1,1)$ (which are the arrows of ${\bf 2}$); the morphisms are  $(0, 0)\to (1, 0)$, $(1, 0)\to (1, 1)$  and $(0, 0)\to (1, 1)$ together with the three identities. Since in ${\bf 2}$ there are two objects we have the following  triples which are not given by identities only: 
$(0, 0)\to (1, 0)=(1, 0)$ and $(1, 0)=(1, 0)\to (1, 1 )$. For $H\in \Hom ({\bf 2}^\square, \cA)$ we have that $H_i(1, 0)\in \cA$ is the only possibly non-zero value, because of Lemma \ref{purity}, and the exact sequences of the triples are
$$H_i(0, 0)\to H_i(1, 0)\by{=} H_i(1, 0)\hspace*{0.5cm}\text{and}\hspace*{0.5cm}H_i(1, 0)\by{=}H_i(1, 0)\to H_i(1, 1 )$$
with $\partial_i=0$. If  $H\in \Hom_{\rm point} ({\bf 2}^\square, \cA)$ we then have that $H_0(1, 0)\in \cA$ is the only possibly non-zero object in such a way that 
$$\Hom_{\rm point} ({\bf 2}^\square, \cA)\cong \Hom_{\rm point} ({\bf 1}, \cA)\cong \cA$$
induced by ${\bf 1}\to {\bf 2}^\square$ sending $*\leadsto (1,0 )$.
Therefore $\Ab_R\cong \iRmod\cong \cA_\partial^{\rm point}({\bf 2})\cong \cA^{\rm point}({\bf 1})$ by Example \ref{pointex}.
Note that the functor ${\bf 2}\to {\bf 2}^\square$ sending $1\leadsto (1,0)$ and $0\leadsto (0, 0 )$  yields 
$$\cA \cong \Hom_{\rm point} ({\bf 2}^\square, \cA)\subsetneq \Hom_{\rm point} ({\bf 2}, \cA)$$
and $r_\partial^{\rm point} :\cA^{\rm point}({\bf 2})\onto \iRmod$. Note that  $\cA \subset \Hom_{\rm point} ({\bf 2}, \cA)$ is determined by those homologies $H$ such that $H_i(0)=0$ for $i\in \Z$, \cf additive homologies defined below.
\end{example}
\begin{example}[Almost trivial homology] \label{almtriv} 
For $\cC$ a category with a strictly initial object $0$ we say that $X = 0$ if there is a morphism $X\to 0$ (note that if such a morphism exists is an isomorphism).  Assume that there is at least one distinguished object $*\neq 0$ as in the previous example and  let ${\bf 2}^\square\to \cC^\square$ be the unique functor induced by $0 \leadsto 0$ and $1\leadsto *$. Let $\cA$ be a non trivial abelian category and pick $A\in \cA$, $A\neq 0$.  The ``almost trivial'' homology $H^!:\cC^{\square}\to\cA$ is given by $H^!_i(X,Y)\df 0$ for $(X,Y)\in \cC^{\square}$, either $Y\neq 0$ and $i\in \Z$ or $Y=0$ and $i\neq 0$; moreover, $H^!_0(X)\df A\in \cA$ for all $X\neq 0$, $H^!_0(0)\df 0$, and $H^!$ is sending morphisms of $\cC^\square$ to identities or zero morphism, \eg for $f: X\to X'$ with $X, X'\neq 0$ $H^!_0(f)$ is the identity of $A$. It is easy to check that this is a relative homology satisfying the point axiom for $\cC$ with $*$ as a point object: by the universal property we then get an exact functor $r_{H^!}:  \cA_\partial^{\rm point}(\cC)\to  \cA$ which is not trivial. Let $\Hom_!(\cC^\square, \cA)\subseteq \Hom (\cC^\square, \cA)$ be the subcategory of almost trivial homologies. We then get $\Hom_{\rm point} ({\bf 2}^\square, \cA)\cong \Hom_!(\cC^\square, \cA)$ by restriction. 
\end{example}

The following is the relative analogue of Proposition \ref{kproj}
 \begin{propose} \label{relkproj}
 Let $\cC$ be a category with a strictly initial object $0$ and a final object $1$. Then for each integer $k\in\Z$ there is a quotient
$$\pi_k:\cA_\partial(\cC)\onto \iRmod$$
 with a section $\iota_k :  \iRmod\to \cA_\partial (\cC)$ (see Definition \ref{unrepdef}, \cf Examples \ref{pointex}-\ref{pointexrel}). For $k=0$ the functor $\pi_0$ factors through a further quotient  $$!^{\rm point}:\cA_\partial^{\rm point}(\cC)\onto \iRmod$$
with a section $!_{\rm point} :  \iRmod\to \cA_\partial^{\rm point}(\cC)$. 
 \end{propose}
 \begin{proof} For each  $X\in \cC$, $X\neq 0$, let $\int : X\to 1$ and get $\int_i : H_i(X)\to H_i(1)$ in $\cA_\partial (\cC)$ where $H$ is the universal relative homology.
For each fixed integer $k$ consider the thick subcategory of $\cA_\partial (\cC)$ generated by $H_i(X)$, $i\neq k$, $\ker \int_k$ and $\coker \int_k$ for $X\neq 0$.  Denote $\cA_\partial (\cC)_k$ the resulting quotient of  $\cA_\partial (\cC)$. Note that, if $Y\neq 0$,  also $H_i(X, Y)$ is vanishing in $\cA_\partial (\cC)_k$ for all $i\in \Z$  by the long exact sequence of the pair; for $i=k$ we have 
$$\xymatrix{
0 \ar[r]^-{}& H_{k+1}(X,Y)\ar[d]^-{\int_{k-1}}\ar[r]^-{\partial_{k+1}} & H_k(Y)\ar[r]^-{\alpha_k}  \ar[d]^-{\int_k} & H_k(X) \ar[d]^-{\int_k}\ar[r]^-{\beta_k} & H_k(X,Y) \ar[r]^-{}  \ar[d]^-{\int_k} &  0 & \\
0 \ar[r]^-{}& H_{k+1}(1,1) \ar[r]^-{\rm zero} & H_k(1)\ar[r]^-{=}  & H_k(1) \ar[r]^-{\rm zero} & H_k(1,1) \ar[r]^-{}  &  0}$$
where the central square is a square of iso in $\cA_\partial (\cC)_k$ and therefore also $H_{k+1}(X,Y)=H_{k}(X,Y)=0$ in 
 $\cA_\partial (\cC)_k$. If $\Hom_k (\cC^\square, - )\subset \Hom (\cC^\square, - )$ is the subfunctor determined by the trivial homology and those non trivial relative homologies $H = \{H_i\}$ such that $H_i (X, Y)\neq 0$ implies $i= k$, $Y=0$, $X\neq 0$ and  $\int_k: H_k (X)\cong H_k(1)\neq 0$, then $\Hom_k (\cC^\square, - )$  is represented by the category  $\cA_\partial(\cC)_k$. Clearly $\cA \cong \Hom_k (\cC^\square, \cA )$ by sending $A  \leadsto H_k(1)\df A$ and $\iRmod\cong \cA_\partial (\cC)_k$. We then get a quotient $\pi_k:\cA_\partial (\cC)\onto \iRmod$ for each integer $k\in \Z$. Now $H_k(1)\in\cA_\partial (\cC)$ yields an exact functor $\iota_k : \iRmod\cong\cA_\partial (\cC)_k\to \cA(\cC)$ whose composition with the projection $\pi_k : \cA(\cC)\to \cA_\partial (\cC)_k\cong \iRmod$ is the identity.
 
For $k=0$, assuming the point axiom we have $H_i(*)=0$ for $i\neq 0$, and we can argue with $\int_0 : H_0(X)\to H_0(1)$ and $H_i(X)$ for $X\in \cC$, $i\neq 0$, in $\cA_\partial^{\rm point}(\cC)$: the last claims are clear.
\end{proof}
\begin{remark}
For $\cC$ a category with a strictly initial object $0$ and a final point object $1$, consider the almost trivial homology $H^!:\cC^{\square}\to\cA_\partial^{\rm point}(\cC)$ of Example \ref{almtriv} given by $H^!_0(X)\df H_0(1)\in \cA_\partial^{\rm point}(\cC)$ for $X\neq 0$. We then get an exact functor $r_{H^!}:  \cA_\partial^{\rm point}(\cC)\to  \cA_\partial^{\rm point}(\cC)$ such that $\cA (H^!) \df \cA_\partial^{\rm point}(\cC)/\ker r_{H^!}$ is equivalent to $\Ab_R$ and $r_{H^!}= !_{\rm point}!^{\rm point}$. 
\end{remark}
\begin{axiom}[Additivity] \label{addax}
Let $H$ be a homology or a relative homology.  Say that  $H$ is \emph{finitely additive} if $$H_i(X)\oplus H_i(X')\by{\simeq} H_i(X\coprod X')$$ for $X, X'\in \cC$ whenever the coproduct $X\coprod X'$ exists in $\cC$, \eg we may assume that the category has finite coproducts. If  $\cC$ has (small) coproducts say that  $H$ with values in a Grothendieck category $\cA$ is \emph{additive} if it is finitely additive and moreover we have an isomorphism 
$$\bigoplus_k{} H_i (X_k)\longby{\simeq}  H_i (\coprod_k X_k)$$ induced by the canonical morphisms $X_k \to\coprod X_k$ of  $\cC$ for an infinite (small) coproduct.  
\end{axiom} 
Denote $\Hom_{\rm add} (\cC^{\square}, \cA )$ the category of relative finitely  additive homologies in $\cA$ and $\Hom_{\rm Add} (\cC^{\square}, \cA )$ the category of additive homologies in a Grothendieck category $\cA$. 
\begin{propose}\label{univadd}
If $\cC$ is a category with finite coproducts then $$\Hom_{\rm add} (\cC^{\square}, - ): \Ex_R\to \Cat$$ is representable by $\cA_\partial^{\rm add}(\cC)$ which is a quotient of $ \cA_\partial(\cC)$. 
If $\cC$ is a category with (small) coproducts then there is a universal additive relative homology 
$H = \{H_i\}_{i\in \Z} :\cC^{\square}\to  \cA_\partial^{\rm Add}(\cC)$
with values in a Grothendieck category $\cA_\partial^{\rm Add}(\cC)$ which is a quotient of $\Ind \cA_\partial(\cC)$ and also a quotient of $\Ind \cA_\partial^{\rm add}(\cC)$.
\end{propose}
\begin{proof} Let $H :\cC^{\square}\to\cA_\partial (\cC)$ be the universal homology. 
Let $ \cA_\partial^{\rm add}(\cC)$  be the quotient of  $ \cA_\partial(\cC)$ by the thick subcategory $\cS$ generated by the kernels and cokernels of $ \bigoplus_kH_i (X_k)\to H_i (\coprod_k X_k)$ for finite coproducts: the image of $H$ under the projection $ \cA_\partial(\cC)\onto \cA_\partial^{\rm add}(\cC)$ yields the universal finitely additive relative homology. 

Using Corollary \ref{induni} we can make the same construction in $\Ind \cA_\partial(\cC)$ considering, for all $i\in \Z$,
kernels and cokernels of $$``\bigoplus_{k\in I}{ } " H_i (X_k)\to H_i (\coprod_{k\in I} X_k)$$ 
where the ``sum'' over a small family $\{X_k\}_{k\in I}$ of objects is defined in \cite[Notation 8.6.1]{KS}: we thus obtain $\cA_\partial^{\rm Add}(\cC)$ as a quotient of $\Ind \cA_\partial(\cC)$ by $\cS$ the thick and localizing subcategory generated by the mentioned kernels and cokernels in such a way that $\Ind \cA_\partial(\cC)\onto \cA_\partial^{\rm Add}(\cC)$ is an $R$-linear exact functor commuting with (filtered) colimits which factors through $\Ind \cA_\partial^{\rm add}(\cC)$. For $K\in\Hom_{\rm Add} (\cC^{\square}, \cA)$ with $\cA$ Grothendieck we obtain $r_K : \Ind \cA_\partial(\cC)\to \cA$ such that $r_K(H)=K$, as in Corollary \ref{induni}, and
 $$\bigoplus_{k}{ } K_i (X_k)=r_K (``\bigoplus_{k}{ } " H_i (X_k)) \by{\simeq} r_K(H_i (\coprod_k X_k))= K_i (\coprod_k X_k)$$ 
so that $r_K $ factors uniquely through the Grothendieck category $\cA_\partial^{\rm Add}(\cC)$ (\cf\cite[Prop.\,1.6]{GG}).
\end{proof}
We also get the refined exact functor 
$$r_\partial^{\rm add} : \cA^{\rm add}(\cC)\to \cA_\partial^{\rm add}(\cC)$$
where  $\cA^{\rm add}(\cC)$ is a quotient of $\cA(\cC)$ by the analogue thick subcategory indicated in the proof of Proposition \ref{univadd}.
 
\section{Universal Grothendieck homological functors} Let $\cC = \cA$ be an abelian category along with the distinguished subcategory given by the objects of $\cA$ but where a morphism is a mono of $\cA$. The category of pairs $\cC^{\square}=\cA^{\square}$ has objects $(A, B)$ where $B \into A$ is a mono of $\cA$ and we clearly have a functor 
$$ (A, B)\leadsto A/B: \cA^{\square}\to \cA$$
which we donote by $q:  \cA^{\square}\to \cA$. Recall that a \emph{Grothendieck homological functor} from $\cA$ to an abelian category $\cB$ is a $\Z$-indexed exact covariant $\partial^*$-functor, see \cite[\S 2.1]{Gr}: this is a family of additive functors that we denote $$T =\{T_i \}_{i\in \Z}: \cA\to \cB$$ together with the boundaries $\partial_i$ decreasing the index, as it is well known; recall that Grothendieck also considers $[a, b]$ an interval in $\Z$ as a set of indexes including $[-\infty , +\infty] =\Z$. Denote by $\Hom_{\partial}(\cA, \cB)$ the category of such a functors plus $\partial$-naturality: we obtain 
$$\Hom_{\partial}( \cA, - ) : \Ex \to \Cat$$
the 2-functor of Grothendieck homological functors.
\subsection{$\partial$-homology}
For $\{T_i \}_{i\in \Z}\in \Hom_{\partial}( \cA, \cB)$, by restriction along $q : \cA^{\square}\to \cA$, set $$H_i^T(A, B)\df T_i(A/B)$$ and get a corresponding family $\{H_i^T\}_{i\in \Z}$ of functors on $\cA^{\square}$. Given a triple $(B, C)\into (A, B)$ in $\cA^{\square}$ we get the following 
$$0\to B/C \to A/C \to A/B\to 0$$
short exact sequence in $\cA$ 
 and for a $\partial$-cube in $\cA^{\square}$ we get a morphism of the corresponding short exact sequences in $\cA$; actually, the functor $q$ lifts to a functor from the category $\cA^{\blacksquare}$ of $\partial$-cubes, see Remark \ref{partialrmk} b), to that of short exact sequences in $\cA$. Therefore, we get the long exact sequence 
$$\cdots \to T_i(B/C)\to T_i(A/C)\to T_i(A/B)\longby{\partial_i} T_{i-1}(B/C)\to \cdots$$
along with the naturality with respect to $\partial$-cubes. Now $H^T\df \{ H^T_i\}_{i\in \Z}\in \Hom_{\rm add} (\cA^{\square},\cB)$ is finitely additive in the sense of Axiom \ref{addax}:  note that $H_i^T(-) = T_i(-):\cA \to \cB$  is an additive functor between abelian categories, \ie one inducing homomorphisms of the ${\rm Hom}$-groups, and this is equivalent to say that  $H_i^T(-) $ commutes with finite sums. Clearly $T\leadsto H^T$ is functorial as a morphism $\varphi : T\to T'$ is assumed to be compatible with the boundaries. 
\begin{defn}\label{parthom}
 A finitely additive relative homology $H\in \Hom_{\rm add} (\cA^{\square},\cB)$ is a \emph{$\partial$-homology} if $H_i(A, B)\by{\simeq} H_i(A/B)$ for all $(A, B)\in \cA^{\square}$.
 \end{defn}
 Note that $H^T$ is a $\partial$-homology. We have:
\begin{lemma} \label{partial}
Let  $\cA$ be an abelian category and let  $\cA^{\square}$ be the category of monos. The functor 
$$T\leadsto H^T : \Hom_{\partial}( \cA, \cB)\into \Hom_{\rm add}(\cA^{\square}, \cB)$$
 is fully faithful with essential image the subcategory of  $\partial$-homologies.
 \end{lemma}
\begin{proof} Since $T$  is the restriction of $H^T$ along $\cA\into \cA^{\square}$ the fully faithfulness is clear.  For a  $\partial$-homology $H= \{H_i \}_{i\in \Z}\in \Hom_{\rm add}(\cA^{\square}, \cB)$  define $T^H =\{T_i ^H\}_{i\in \Z }$ by considering the restricted homology  $A\leadsto T_{i}^H(A)\df H_{i} (A)$: this yields additive functors as $H$ is additive. If  $0\to B\to A\to C\to 0$ is exact in $\cA$, $C= A/B$ and $H_i (A, B)\by{\simeq} H_i(A/B)$, then 
$$\cdots\to T_{i}^H(B)\to T_{i}^H(A)\to T_i^H(A/B)\longby{\partial_i}T_{i-1}^H(B)\to \cdots$$ is given by the long exact sequence of the pair $(A,B)\in \cA^\square$ for the relative homology $H$. Naturality of  $T^H$ is also given by naturality of $H$.  We thus obtain $T'\df T^H\in \Hom_{\partial}( \cA, \cB)$ such that $H^{T'}=H$.  \end{proof}
\begin{thm}[Universal $\partial$-homology]\label{unipart}
For $\cA$ an abelian category the universal $\partial$-homology exsits. Therefore, the universal Grothendieck homology $T^\partial :  \cA\to \cA^\partial$ exists, \ie the functor $\Hom_{\partial } (\cA, - ): \Ex \to \Cat$  is representable by $( \cA^\partial, T^\partial)$.
\end{thm}
\begin{proof}
 The universal Grothendieck homology $T^\partial $ is given by the universal  $\partial$-homology $H^\partial$ if  $\Hom_{\partial } (\cA, - )$ is regarded as a subfunctor of $\Hom_{\rm add}(\cA^{\square}, -)$ given by $\partial$-homologies, by Lemma \ref{partial}.  Let  $\cA^\partial$ be the quotient of $ \cA_\partial^{\rm add}(\cA)$ given by the thick subcategory generated by kernels and cokernels of $ H_i^{\rm add}(A, B)\to H_i^{\rm add}(A/B)$ for all $(A, B)\in \cA^{\square}$, where $ H^{\rm add}: \cA^{\square}\to \cA_\partial^{\rm add}(\cA)$ is the universal additive relative homology of Proposition \ref{univadd}. Then the image $H^\partial: \cA^{\square}\to\cA^\partial$ of  $H^{\rm add}$ under the projection is a $\partial$-homology which yields the universal Grothendieck homology  $T^\partial\df T^{H^\partial}$ by the argument in the proof of Lemma \ref{partial}.
\end{proof}
\begin{remark} To see that $\cA^\partial\neq 0$ for any $\cA\neq 0$ we point out that $\Ex (\cA, -)$  is a subfunctor of $\Hom_{\partial } (\cA, - )$ given by
$F\leadsto T_a\df F$ and $T_i \df 0$ for $i\neq a$ and $a\in \Z$ fixed. Moreover, for any object $A\in \cA$ of an abelian category $\cA$ such that $\cA\neq \cA\otimes\Q$, \eg $\cA =\Zmod$,   let ${}_nA$ and $A/n$ be the kernel and cokernel of the multiplication by $n$, a fixed positive integer. Let $T_i^n(A)=0$ for $i\neq a, a-1$, $T_a^n(A) \df{}_nA$ and $T_{a-1}^n(A)\df A/n$  for $a\in \Z$ fixed. We have $\{T_i^n\}_{i\in \Z}\in \Hom_{\partial } (\cA, \cA)$ 
\end{remark}
\subsection{Satellite homology}
Consider $T =\{T_i \}_{i\in \Z }$ with $T_{-i} = 0$ for $i<0$ and set $T^i\df T_{-i}$ now concentrated in positive degrees $i\geq 0$, with the $\partial_i$ increasing degrees: thus  $T^0$ is left exact (this is an exact covariant $\partial$-functor in \cite[\S 2.1]{Gr}). We shall denote by $ \Hom_{\partial }^+  (\cA, \cB )\subset \Hom_{\partial}( \cA, \cB)$ the subcategory of such $\partial$-functors.  
For a left exact functor $F: \cA \to \cB$ denote $RF =\{ R^iF\}_{i\in \N}$ the \emph{right satellite} or right derived functor which belongs to $\Hom_{\partial }^+  (\cA, \cB )$: it exists under the condition that every object of $\cA$ has an injective effacement, see \cite[Thm.\,2.2.2 \& \S 2.3]{Gr} and \cite[Chap.\,III]{CE}, \eg if $\cA$ is Grothendieck, see also \cite[Thm.\,9.6.2]{KS}. 

Considering additive homologies on $\cA^\square$ concentrated in non positive degrees $[-\infty, 0]$, \ie $H =\{H_{i}\}_{i \in \Z}$ such that $H_i =0$ if  $i>0$, we shall denote $H =\{H^{i}\df H_{-i}\}_{i \in \N}$ and
$\Hom_{\rm add}^+(\cA^{\square}, \cB)\subset \Hom_{\rm add}(\cA^{\square}, \cB)$ the corresponding subcategory. Now 
 observe that for a $\partial$-homology $H$ the restricted functor  $H^0:\cA\to \cB$ (\cf the proof of Lemma \ref{partial}) is a left exact functor. We may consider the right satellite $RH^0$ of $H^0$ and get another $\partial$-homology. 
\begin{lemma}\label{lex}
Let $\cA$ be an abelian category. If the right satellite $RF$ exists for any left exact functor $F\in  \Lex (\cA, \cB)$   (\eg if  $\cA$ is a  Grothendieck category) then the forgetful functor
$$ T =\{T^i\}_{i\in \N }\leadsto T^0:   \Hom_{\partial }^+  (\cA, \cB ) \to \Lex (\cA, \cB)$$
has a left adjoint/left inverse given by the right satellite 
$$F\leadsto RF : \Lex (\cA, \cB)\into \Hom_{\partial }^+  (\cA, \cB )$$
which is fully faithful. The fully faithtful functor
$$ \Lex (\cA, \cB)\into \Hom_{\rm add}^+(\cA^{\square}, \cB)$$
(induced  by composition with that in Lemma \ref{partial}) has essential image those $\partial$-homolo\-gies $H$ which coincide with the right satellite of their $H^0$. 
\end{lemma}
\begin{proof} If $RF$ exists then ${\rm Nat} (F, T^0) = {\rm Nat}_\partial (RF, T)$, where ${\rm Nat}$ are natural transformations and ${\rm Nat}_\partial$ are morphisms of $\partial$-functors. This follows essentially by definition, see \cite[\S 2.2]{Gr}: we have that any $\varphi^0: F\to T^0$ extends to a unique $\varphi : RF\to T$ morphism of $\partial$-functors. 
Since the restricted $\partial$-homology $H$ (denoted $T^H$ in the proof of Lemma \ref{partial}) is a $\partial$-functor we thus obtain a unique morphism $\varphi_H : RH^0\to H$. Note that for $H = H^{RF}$ we get  $\varphi_{H^{RF}} = id_{H^{RF}}$ as $RF$ is the restriction of $H^{RF}$. Therefore the essential image of  $F\leadsto H^{RF}$ is given by those $\partial$-homologies that $\varphi_H : RH^0\by{\simeq} H$ is an isomorphism. 
 \end{proof}
 \begin{defn}
Call a $\partial$-homology $H\in \Hom_{\rm add }^+  (\cA^\square, \cB )$ a \emph{right satellite homology} if  the canonical morphism $\varphi_H : RH^0\by{\simeq} H$ is an isomomorphism. 
\end{defn}
From Lemma \ref{lex} we have a canonical equivalence of the category of left exact functors with  right satellite homologies. The analogue of Theorem \ref{unipart} for satellite homologies is the following:
\begin{thm}[Universal right satellite homology]\label{unisat}
Let $\cA$ be an abelian category. If the right satellite $RF$ exists for any left exact functor $F\in  \Lex (\cA, \cB)$  then the universal right satellite homology exists and it is given as follows:
\begin{itemize}
\item [{\it i)}] the universal left exact functor $F^{\rm lex} : \cA \to \cA^{\rm lex}$ exists, where $\cA^{\rm lex}$ is a quotient of Freyd's category $\Ab (\cA)$, \ie $\cA^{\rm lex}$ is representing $\Lex (\cA, - ): \Ex \to \Cat$;
\item [{\it ii)}]  the right satellite $RF^{\rm lex}  = \{R^iF^{\rm lex} \}_{i\in\N}$ provides the universal right satellite homology under the equivalence in Lemma \ref{lex};
\item [{\it iii)}] the universal property is the following: for $RF=\{RF^i\}_{i\in \N }\in \Hom_{\partial} ^+(\cA, \cB )$ any right derived functor of a left exact functor $F$ we obtain an exact functor $r_{RF}: \cA^{\rm lex} \to \cB$ (unique up to isomorphism) and a factorisation
$$\xymatrix{\cA\ar[r]^{RF^{\rm lex}} \ar[d]_{RF} & \cA^{\rm lex} \ar[dl]^{r_{RF}} \\  \cB}$$
such that $r_{RF}R^iF^{\rm lex} =R^iF$ and $r_{RF}(\partial_i)= \partial_i$ for all $i\in \N$; 
\item [{\it iv)}] we have a factorisation
$$\xymatrix{D^+(\cA)\ar[r]^{RF^{\rm lex}} \ar[d]_{RF} &D^+( \cA^{\rm lex}) \ar[dl]^{r_{RF}} \\ D^+(\cB)}$$
of triangulated derived functors. 
\end{itemize}
\end{thm}
\begin{proof}
i) Consider $Y\df |\cA| :\cA \into\Ab (\cA)$ the Freyd's free abelian category of $\cA$ regarded as an additive category, see Theorem \ref{Freyd}. Consider 
$0\to B\to A\to C\to 0$ an exact sequence in $\cA$ and let $0\to Y(B)\to Y(A)\to Y(C)\to 0$ be the induced complex in $\Ab (\cA)$. Denote $H_B$ and $H_A$ the homologies of this complex at $Y(B)$ and $Y(A)$ respectively.  Let $\cS$ be thick subcategory of $\Ab (\cA)$ generated by the objects $H_B$ and $H_A$ for all $B\into A$ mono of $\cA$. Define $\cA^{\rm lex}\df \Ab (\cA)/\cS$
and let $F^{\rm lex} : \cA \to \cA^{\rm lex}$ be the composition of $Y$ with the projection in such a way that the 
$0\to F^{\rm lex} (B)\to F^{\rm lex} (A)\to F^{\rm lex}(C)$ is now exact in  $\cA^{\rm lex}$.  The left exact functor $F^{\rm lex}$ is the universal left exact functor as it easily follows from the universality of Freyd's construction. 

ii) It follows from i) and the proof of  Lemma \ref{lex}-\ref{partial} since the right satellite homology 2-functor is representable by $RF^{\rm lex}$ regarded in $\Hom_{\rm add }^+  (\cA^\square, \cA^{\rm lex} )$.

iii) It is given by representability in ii).

iv) Since $F = r_{RF}F^{\rm lex}$ and $r_{RF}$ is exact.
\end{proof}
\begin{remark}
Note that there is an exact functor $F_{id_\cA}: \Ab (\cA)\to \cA$ whose restriction to $\cA$ is the identity. The functor 
$F_{id_\cA}$ factors though $\cA^{\rm lex}$ and yields an exact functor $r_{id_{\cA}}: \cA^{\rm lex}\to\cA$ which is a section of the left exact functor $F^{\rm lex}: \cA \to\cA^{\rm lex}$. 
\end{remark} 
\begin{example}
Let $\cR$ be an additive category such that $\cR\mbox{-}{\rm mod}$ is abelian, \eg $\Rmod$ for  $R$ a coherent ring. 
For $\cA = \cR\mbox{-}{\rm mod}^{op}$ we get $\cA^{\rm lex} = \Ab (\cR)$ and $F^{\rm lex}= Y$ is given by the Yoneda left exact embedding as shown (\cf  \cite[Example 4.9 ]{P}). For  $F$ left exact 
$$\xymatrix{\cR\mbox{-}{\rm mod}^{op}\ar[r]^{Y} \ar[d]_{F} & \Ab (\cR) \ar[dl]^{r_{F}} \\  \cB}$$
where $r_{F}$ is given by the universal property of $\Ab (\cR)$ applied to the restriction of $F$ to $\cR$ along the Yoneda  $\cR\into \cR\mbox{-}{\rm mod}^{op}$. In particular, for $\cA= \Zmod^{op}$ we obtain $\cA^{\rm lex} = \iZmod$. 
\end{example}


\begin{thebibliography}{}
\bibitem[SGA4]{sga4}{\sc M. Artin, A. Grothendieck, J.-L. Verdier}: {\it Th\'eorie des topos et cohomologie \'etale des sch\'emas}, S\'eminaire de g\'eom\'etrie alg\'ebrique du Bois-Marie 1963–1964 (SGA 4), Vol. 1, Lect. Notes in Math. {\bf 269}, Springer, 1972.
\bibitem{A} Y. Andr\'e: Pour une th\'eorie inconditionnelle des motifs, {\it Publ. Math.} IH\'ES, {\bf 83}(1996) 5-49.
\bibitem{AB} Y. Andr\'e: Groupes de Galois motiviques et p\'eriodes, Expos\'e 1104 in {\it S\'eminaire Bourbaki: Volume 2015/2016 Expos\'es 1104–1119} Ast\'erisque {\bf 390} (2017) 1-26
\bibitem{AK} Y. Andr\'e \& B. Kahn: Construction inconditionnelle de groupes de Galois motiviques, C. R. Acad. Sci. Paris {\bf 334} (2002) 989-994.
\bibitem{Ay} J. Ayoub: L’alg\`ebre de Hopf et le groupe de Galois motiviques d’un corps de caract\'eristique nulle, Journal f\"ur die reine und angew. Mathematik {\bf 693} (2014), partie I :1-149 ; partie II : 151-226.
\bibitem{UCTII} L. Barbieri-Viale: On topological motives, arXiv:2107.07993 [math.AG].  
\bibitem{BV} L. Barbieri-Viale: $\mathbb{T}$-motives, {\it J. Pure Appl. Algebra}  {\bf 221} (2017) 1495-1898.
\bibitem{BVP} L. Barbieri-Viale: \& M. Prest: Definable categories and $\mathbb{T}$-motives,  {\it Rend. Sem. Mat. Univ. Padova} {\bf 139} (2018) 205-224.
\bibitem{BVHP} L. Barbieri-Viale, A. Huber-Klawitter \& M. Prest: Tensor structure for Nori motives, {\it Pacific Journal of Math.} {\bf 306} No. 1 (2020) 1-30
\bibitem{BVPT} L. Barbieri-Viale \& M. Prest: Tensor product of motives via K\"unneth formula, {\it J. Pure Appl. Algebra} {\bf 224} No. 6 (2020)  
\bibitem{BCL} L. Barbieri-Viale, O. Caramello \& L. Lafforgue: Syntactic categories for Nori motives, {\it Sel. Math. New Ser.} {\bf 24} (2018) 3619-3648.
\bibitem{Be} A. Beilinson: Remarks on Grothendieck's standard conjectures in {\it Regulators} Contemp. Math. {\bf 571} AMS, Providence, (2012) 25–32
\bibitem{CE} H. Cartan \& S. Eilenberg: {\it Homological Algebra} Princeton University Press, (1956).
\bibitem{DT} P. Deligne \& J.S. Milne: Tannakian Categories, in {\it Hodge Cycles, Motives, and Shimura Varieties}, LNM 900, 1982, pp. 101-228
\bibitem{De} P. Deligne: Le groupe fondamental de la droite projective moins trois points, in {\it Galois groups over $\Q$} (Berkeley, CA, 1987), Math. Sci. Res. Inst. Publ., {\bf 16} Springer, New York, (1989) 79-297 
\bibitem{Dy} E. Dyer: {\it Cohomology theories}, Mathematics Lecture Note Series W. A. Benjamin, Inc., New York-Amsterdam,  1969.
\bibitem{DU} D. Dugger: Universal homotopy theories, Advances in Math. {\bf 164} (2001) 144-176 
\bibitem{ES}  S. Eilenberg \& N. Steenrod: Axiomatic approach to homology theory, {\it Proc. Natl. Acad. Sci.} USA {\bf 31 } (1945)
\bibitem{ESF}  S. Eilenberg \& N. Steenrod: {\it Foundations of Algebraic Topology} Princeton Legacy Library, 2016 (1952)
\bibitem{Fr} P. Freyd: Representations in Abelian categories, in: Eilenberg S., Harrison D.K., MacLane S., R\"ohrl H. (eds) {\it Proceedings of the Conference on Categorical Algebra, La Jolla, 1965} Springer (1966).
\bibitem{GG} G. Garkusha: Grothendieck categories, {\it St. Petersburg Math. J.} {\bf 13} (2002) 149-200
\bibitem{GA} G. Garkusha \& I. Generalov: Duality for categories of finitely presented modules, {\it St. Petersburg Math. J.} {\bf 11} (2000) 1051-1061 
\bibitem{Gr} A. Grothendieck: Sur quelques points d'alg\'ebre homologique, {\it Tohoku Math. Journal}, Second Series, {\bf 9} (1957)
\bibitem{HMS} A. Huber  \&  S. M\"uller-Stach: {\it Periods and Nori motives} With contributions by Benjamin Friedrich and Jonas von Wangenheim. Ergebnisse der Mathematik und ihrer Grenzgebiete. 3. Folge. A Series of Modern Surveys in Mathematics [Results in Mathematics and Related Areas. 3rd Series. A Series of Modern Surveys in Mathematics] {\bf 65}, Springer, Cham, 2017.
\bibitem{AH} A. Heller: Stable homotopy categories, {\it Bull. Amer. Math. Soc.} {\bf 74} (1968)  28-63 
\bibitem{JL} C. U. Jensen and H. Lenzing: {\it Model Theoretic Algebra; with particular emphasis on Fields, Rings and Modules} Gordon and Breach, 1989.
\bibitem{El} P. Johnstone: {\it Sketches of an Elephant: A Topos Theory Compendium} Vol. 1 \& 2, Clarendon Press, Oxford Logic Guides Vol. 43 \& 44, 2002.
\bibitem{KS} M. Kashiwara \& P. Schapira: {\it Categories and Sheaves} Grundlehren Math. Wiss., vol. 332, Springer, 2006.
\bibitem{KH} H. Krause: {\it Homological Theory of Representations} to appear, Cambridge Studies in Advanced Mathematics, Cambridge University Press, Cambridge, 2021.
\bibitem{NT} A. Neeman: {\it Triangulated categories}, Annals of Math. Studies Vol. 148, Princeton Univ. Press, 2001.
\bibitem{P} M. Prest: {\it Definable additive categories: purity and model theory} Mem. Amer. Math. Soc. {\bf 210} (2011)
\bibitem{PR} M. Prest \& R. Rajani: Structure sheaves of definable additive categories {\it J. Pure Appl. Algebra} {\bf 214} (2010) 1370-1383
\bibitem{Vh} V. Voevodsky: Homology of Schemes,  {\it Selecta Mathematica} Vol. 2  (1996) 111-153
\end{thebibliography}
\end{document}